\documentclass{article}

\usepackage[margin=1in]{geometry}
\usepackage{amsfonts,amsmath,amssymb,amsthm}
\usepackage{algorithm}
\usepackage{hyperref}
\usepackage{cleveref}
\newtheorem{theorem}{Theorem}[section]
\newtheorem{lemma}[theorem]{Lemma}
\newtheorem{proposition}[theorem]{Proposition}
\newtheorem{corollary}[theorem]{Corollary}
\newtheorem{fact}[theorem]{Fact}
\newtheorem{remark}[theorem]{Remark}
\numberwithin{equation}{section}
\title{The Fundamental Subspaces of \\Ensemble Kalman Inversion}

\author{Elizabeth Qian\thanks{School of Aerospace Engineering, School of Computational Science and Engineering, Georgia Institute of Technology, Atlanta, GA 
  (\texttt{eqian@gatech.edu}, \href{https://www.elizabethqian.com}{\texttt{https://www.elizabethqian.com}}).}
\and Christopher Beattie\thanks{Department of Mathematics, Virginia Polytechnic Institute and State University 
  (\texttt{beattie@vt.edu})}}

\usepackage{enumitem}
\usepackage{mathtools}
\usepackage{amsfonts}
\usepackage{graphicx}
\usepackage{epstopdf}
\usepackage{longtable}
\usepackage[noend]{algorithmic}
\usepackage{bbm}
\ifpdf
  \DeclareGraphicsExtensions{.eps,.pdf,.png,.jpg}
\else
  \DeclareGraphicsExtensions{.eps}
\fi

\usepackage{amsopn}
\DeclareMathOperator{\diag}{diag}

\newcommand{\bbM}{\mathbf{M}}
\newcommand{\bbN}{\mathbf{N}}
\newcommand{\bbP}{\mathbf{P}}
\newcommand{\bbQ}{\mathbf{Q}}
\newcommand{\R}{\mathbb{R}}
\newcommand{\E}{\mathbb{E}}
\newcommand{\cov}{\mathbb{C}\text{ov}}

\newcommand{\sfE}{\textsf{E}}
\newcommand{\sfCov}{\textsf{cov}}
\newcommand{\rank}{\textsf{rank}}
\newcommand{\Ran}{\textsf{Ran}}
\newcommand{\Ker}{\textsf{Ker}}

\newcommand{\bGamma}{\boldsymbol{\Gamma}}
\newcommand{\bDelta}{\boldsymbol{\Delta}}
\newcommand{\bSigma}{\boldsymbol{\Sigma}}
\newcommand{\bcalM}{\boldsymbol{\mathcal{M}}}
\newcommand{\beps}{\boldsymbol{\varepsilon}}

\newcommand{\btheta}{\boldsymbol{\theta}}
\newcommand{\bomega}{\boldsymbol{\omega}}

\newcommand{\bcP}{\boldsymbol{\mathcal{P}}}
\newcommand{\bcQ}{\boldsymbol{\mathcal{Q}}}
\newcommand{\bcN}{\boldsymbol{\mathcal{N}}}

\newcommand{\cN}{\mathcal{N}}

\newcommand{\bH}{\mathbf{H}}
\newcommand{\bI}{\mathbf{I}}
\newcommand{\bC}{\mathbf{C}}
\newcommand{\tbB}{\tilde{\mathbf{B}}}
\newcommand{\tbC}{\tilde{\mathbf{C}}}
\newcommand{\bD}{\mathbf{D}}
\newcommand{\bV}{\mathbf{V}}
\newcommand{\bU}{\mathbf{U}}
\newcommand{\bW}{\mathbf{W}}
\newcommand{\bK}{\mathbf{K}}

\newcommand{\by}{\mathbf{y}}
\newcommand{\bw}{\mathbf{w}}
\newcommand{\bv}{\mathbf{v}}
\newcommand{\bu}{\mathbf{u}}
\newcommand{\bh}{\mathbf{h}}
\newcommand{\ba}{\mathbf{a}}
\newcommand{\bb}{\mathbf{b}}
\newcommand{\br}{\mathbf{r}}

\newcommand{\fisher}{\bH^\top\bGamma^{-1}\bH}

\makeatletter
\def\mathcenterto#1#2{\mathclap{\phantom{#1}\mathclap{#2}}\phantom{#1}}
\let\old@widetilde\widetilde
\def\widetildeto#1#2{\mathcenterto{#2}{\old@widetilde{\mathcenterto{#1}{#2\,}}}}
\let\old@widehat\widehat
\def\widehatto#1#2{\mathcenterto{#2}{\old@widehat{\mathcenterto{#1}{#2\,}}}}
\makeatother

\ifpdf
\hypersetup{
  pdftitle={The Fundamental Subspaces of Ensemble Kalman Inversion},
  pdfauthor={E. Qian and C. Beattie}
}
\fi

\begin{document}

\maketitle

\begin{abstract}
Ensemble Kalman Inversion (EKI) methods are a family of iterative methods for solving weighted least-squares problems, especially those arising in scientific and engineering inverse problems in which unknown parameters or states are estimated from observed data by minimizing the weighted square norm of the data misfit. Implementation of EKI requires only evaluation of the forward model mapping the unknown to the data, and does not require derivatives or adjoints of the forward model. The methods therefore offer an attractive alternative to gradient-based optimization approaches in inverse problem settings where evaluating derivatives or adjoints of the forward model is computationally intractable. This work presents a new analysis of the behavior of both deterministic and stochastic versions of basic EKI for linear observation operators, resulting in a natural interpretation of EKI's convergence properties in terms of ``fundamental subspaces” analogous to Strang's fundamental subspaces of linear algebra. Our analysis directly examines the discrete EKI iterations instead of their continuous-time limits considered in previous analyses, and provides spectral decompositions that define six fundamental subspaces of EKI spanning both observation and state spaces. This approach verifies convergence rates previously derived for continuous-time limits, and yields new results describing both deterministic and stochastic EKI convergence behavior with respect to the standard minimum-norm weighted least squares solution in terms of the fundamental subspaces. Numerical experiments illustrate our theoretical results.
\end{abstract}

\section{Introduction}
Ensemble Kalman Inversion (EKI) methods describe a family of iterative methods for solving weighted least-squares problems, particularly those arising in \textit{inverse problems}:
 let $\bH:\R^d\to\R^n$ be a \textit{forward operator} mapping $\bv\in\R^d$, representing an unknown parameter or state of a system of interest, to observations $\by\in\R^n$, and let $\bGamma\in\R^{n\times n}$ be a symmetric positive definite weight matrix. Given observed data $\by$, we seek to infer the unknown state $\bv$ by solving:
\begin{align}\label{eq: least squares problem}
    \min_{\bv\in\R^d} (\by - \bH(\bv))^\top \bGamma^{-1} (\by - \bH(\bv)) = \min_{\bv\in\R^d} \|\by - \bH(\bv)\|_{\bGamma^{-1}}^2.
\end{align}
Inverse problems arise in many disciplines across science, engineering, and medicine, including earth, atmospheric, and ocean modeling~\cite{iglesias2014well,majumder2021freshwater,nguyen2012spatial,panteleev2015adjoint,twomey2013introduction}, medical imaging~\cite{chung2010numerical,epstein2007introduction,scherzer2006mathematical}, robotics and autonomy~\cite{palanthandalam2008wind}, and more. In large-scale scientific and engineering applications, solving~\eqref{eq: least squares problem} using standard gradient-based optimization methods may not be feasible because derivatives or adjoints of the forward operator $\bH$ are unavailable or prohibitively expensive to compute.
In contrast, EKI methods can be implemented in an \textit{adjoint-/derivative-free} way, making EKI an attractive alternative to gradient-based methods for solving~\eqref{eq: least squares problem}. Historically, EKI methods have been developed in the contexts of reservoir simulation~\cite{chen2012EnsembleRandomizedMaximum,emerick2013EnsembleSmootherMultiple,emerick2013InvestigationSamplingPerformance,evensen2018AnalysisIterativeEnsemble,gu2007IterativeEnsembleKalman,li2009IterativeEnsembleKalman} and weather and climate modeling~\cite{bocquet2020BayesianInferenceChaotic,cleary2021CalibrateEmulateSample,dunbar2021CalibrationUncertaintyQuantification,gottwald2021SupervisedLearningNoisy,pulido2018StochasticParameterizationIdentification,schneider2017EarthSystemModeling}. 

\begin{algorithm}[t]
\caption{Basic Ensemble Kalman Inversion (EKI)}\label{alg:EKI}
\begin{algorithmic}[1]
\item[]\hspace{-\algorithmicindent}\parbox[t]{\dimexpr\linewidth-\algorithmicindent}{\textbf{Input:}  forward operator $\bH:\R^d\to\R^n$,
initial ensemble $\{\bv_0^{(1)},\ldots,\bv_0^{(J)}\}\subset\R^d$, positive-definite observation misfit weighting $\bGamma\in\R^{n\times n}$,
observations $\by\in\R^n$, observation perturbation covariance $\bSigma\in\R^{n\times n}$.}
\vspace{0.5em}
\FOR{$i = 0,1,2,\ldots,$}
\STATE{Compute observation-space ensemble: $\bh_i^{(j)} = \bH\big(\bv_i^{(j)}\big),$  $j = 1,2,\ldots,J$.}
\STATE{Compute empirical covariances: $\sfCov[\bv_i^{(1:J)},\bh_i^{(1:J)}]$} and $ \sfCov[\bh_i^{(1:J)}]$ \\ 
\STATE{Compute ensemble Kalman gain:  $\bK_i = \sfCov[\bv_i^{(1:J)}\!,\bh_i^{(1:J)}]\cdot \big(\sfCov[\bh_i^{(1:J)}] + \bGamma\big)^{-1}$}
\STATE{Sample $\beps_i^{(j)}$ i.i.d.\ from $\cN(\boldsymbol{0},\bSigma)$ for $j = 1,2,\ldots,J$.}
\STATE{Perturb observations: set $\by_i^{(j)}=\by + \beps_i^{(j)}$ for $j = 1,2,\ldots,J$.}
\STATE{Compute particle update: $\bv_{i+1}^{(j)} = \bv_i^{(j)}+\bK_i(\by_i^{(j)}-\bH\bv_i^{(j)})$ for $j = 1,2,\ldots,J$.
}
\IF{converged}
\STATE{\textbf{return} current ensemble mean, $\sfE[\bv_{i+1}^{(1:J)}]$}
\ENDIF
\ENDFOR 
\end{algorithmic}
\fbox{
{\small In \textit{stochastic} EKI, $\bSigma$ is typically chosen to be $\bSigma=\bGamma$. In \textit{deterministic} EKI, set $\bSigma=\boldsymbol{0}$.}
}
\end{algorithm}
We focus on a basic version of EKI from~\cite{iglesias2013ensemble} and presented in \Cref{alg:EKI}, noting that other EKI methods can be viewed as variations on this theme. Throughout this work, $\E$ and $\cov$ denote the \textit{true} expected value and covariance, respectively; whereas $\sfE$ and $\sfCov$ denote the \textit{empirical} (sample) mean and covariance operators, respectively: 
given $J\in\mathbb{N}$ samples $\{\ba^{(j)}\}_{j=1}^J$ and $\{\bb^{(j)}\}_{j=1}^J$, we define $\sfE[\ba^{(1:J)}] = \frac1J\sum_{j=1}^J \ba^{(j)}$, and
\begin{align*}
    \sfCov[\ba^{(1:J)},\bb^{(1:J)}] = \frac1{J-1}\sum_{j=1}^J (\ba^{(j)}-\sfE[\ba^{(1:J)}])(\bb^{(j)}-\sfE[\bb^{(1:J)}])^\top,
\end{align*}
and $\sfCov[\ba^{(1:J)}] = \sfCov[\ba^{(1:J)},\ba^{(1:J)}]$.
\Cref{alg:EKI} prescribes the evolution of an ensemble of $J$ particles, $\{\bv_i^{(1)},\ldots,\bv_i^{(J)}\}$, initialized at $i = 0$ in some way, e.g., by drawing from a suitable prior distribution,
and subsequently updated for $i = 1,2,3,$ etc. 
Those familiar with ensemble Kalman \textit{filtering} methods will recognize familiar elements in \Cref{alg:EKI}. Indeed, one way to obtain \Cref{alg:EKI} is to apply the ensemble Kalman filter to a system whose dynamics are given by the identity map in the ``forecast" step of the filter. 
The connection to Kalman filtering gives rise to two versions of the method: in Step 6 of \Cref{alg:EKI}, the observations can either be perturbed by random noise\footnote{\Cref{alg:EKI} adds these perturbations to the observed \textit{data}, $\by$, in line with both the EKI literature as well as historical ensemble Kalman filtering literature. Note that after substituting $\by_i^{(j)}=\by+\beps_i^{(j)}$ into the update in Step 7,  $\beps_i^{(j)}$ can alternatively be viewed as a perturbation to the \textit{forecast observations}, $\bH\bv_i^{(j)}$~\cite{hodyssEnsembleStateEstimation2011,vetra-carvalhoStateoftheartStochasticData2018}. In the filtering context, the work~\cite{van2020consistent} shows that this alternative interpretation is more consistent with the underlying statistical inference problem.}, yielding Stochastic EKI (also called the `noisy'~\cite{schillings2018convergence,bungert2023complete} or the `perturbed observations'~\cite{zhang2010ensemble,van2020consistent} case), or the observation perturbation can be set to zero, yielding Deterministic EKI (also called the `noise-free' or `clean data' case~\cite{schillings2017analysis,bungert2023complete}). 
In the ensemble Kalman \textit{filter}, these stochastic perturbations ensure unbiased estimates of the filtering statistics.
In ensemble Kalman \textit{inversion}, numerical evidence suggests that the stochastic perturbations improve performance in nonlinear inverse problems~\cite{iglesias2013ensemble}.

There is a very rich body of literature developing convergence theory for EKI, including convergence analyses of either mean-field limits of the iteration (equivalent to an infinitely large ensemble)~\cite{calvello2022ensemble,ding2021ensemble}, or continuous-time limits, in which the deterministic iteration becomes a system of coupled ordinary differential equations (ODEs)~\cite{bungert2023complete,schillings2017analysis,schillings2018convergence} and the stochastic iteration becomes a system of coupled stochastic differential equations (SDEs)~\cite{blomker2018strongly,blomker2019well,blomker2022continuous}).
The O/SDE solutions at $t=1$ can be shown to approximate a Bayesian posterior distribution, and some works derive sampling methods by numerically integrating the O/SDE from $t=0$ to $t=1$. In contrast, as $t\to\infty$, solutions to the O/SDEs can be shown to converge to the solution of the least-squares problem~\eqref{eq: least squares problem}. In this paper, our focus is on this latter case, in which the basic version of EKI in \Cref{alg:EKI} can be viewed as a discretization from $t=0$ to $t=\infty$ of the underlying O/SDE with unit time step. 

Many works have also introduced methodological variations from the basic algorithm, for example by incorporating a Tikhonov regularization term into the least-squares objective function~\cite{chada2020tikhonov,weissmannAdaptiveTikhonovStrategies2022}, enforcing constraints in the optimization~\cite{albersEnsembleKalmanMethods2019,carrilloConsensusbasedOptimizationEnsemble2023,chadaIncorporationBoxConstraintsEnsemble2019,hanuEnsembleKalmanInversion2024}, or introducing hierarchical~\cite{chadaAnalysisHierarchicalEnsemble2018}, multilevel~\cite{weissmann2022multilevel}, and parallel~\cite{zhangParallelEnsembleKalman2024} versions of the algorithm.
Beyond the successful use of EKI for solving diverse inverse problems in the physical sciences, e.g., in geophysical~\cite{basSwaveVelocityEstimation2022,pensoneaultEnsembleKalmanInversion2023,tsoEnsembleKalmanInversion2024} and biological contexts~\cite{iglesiasEnsembleKalmanInversion2022}, EKI has also been used as an optimizer for training machine learning models~\cite{guth14EnsembleKalman2022,kovachkiEnsembleKalmanInversion2019,lopez-gomezTrainingPhysicsBasedMachineLearning2022,schneiderLearningStochasticClosures2021,schneiderEnsembleKalmanInversion2022}.
In particular, the use of EKI for training neural networks~\cite{kovachkiEnsembleKalmanInversion2019} has motivated the development of EKI variants based on ideas used for gradient-based training of neural networks, including dropout~\cite{liu2023dropout}, data subsampling (also called `(mini-)batching')~\cite{hanuSubsamplingEnsembleKalman2023,kovachkiEnsembleKalmanInversion2019}, adaptive step sizes~\cite{chada2022convergence}, and convergence acceleration with Nesterov momentum~\cite{kovachkiEnsembleKalmanInversion2019}. As we will describe in more detail in the body of this work, the convergence of basic EKI is slow, occurring at a $1/\sqrt{i}$ rate~\cite{blomker2019well,schillings2017analysis,schillings2018convergence}, and convergence occurs only in the subspace spanned by the initial ensemble~\cite{iglesias2013ensemble}, which can be a limitation in settings where the ensemble size is small due to computational cost constraints. Several works have therefore proposed strategies which break the subspace property, including variance inflation or `localization' strategies~\cite{armbrusterStabilizationContinuousLimit2022,blomker2019well,chada2022convergence,huang2022EfficientDerivativefreeBayesian,tong2023localized,liu2023dropout,ghattas2024NonAsymptoticAnalysisEnsemble}, which may involve random perturbations of ensemble members~\cite{huang2022EfficientDerivativefreeBayesian,liu2023dropout,schillings2017analysis}, or adaptive ensemble sizes~\cite{parzerConvergenceRatesAdaptive2022}. Variance inflation can also lead to faster convergence~\cite{chada2022convergence,huang2022EfficientDerivativefreeBayesian}; in particular, perturbing EKI states with appropriately defined noise yields exponential convergence~\cite{huang2022EfficientDerivativefreeBayesian}. 
Beyond the weighted least squares problem~\eqref{eq: least squares problem}, EKI has been adapted to new settings including Bayesian inverse problems~\cite{huang2022EfficientDerivativefreeBayesian}, rare event sampling~\cite{wagnerEnsembleKalmanFilter2022}, and time-varying forward models~\cite{weissmannEnsembleKalmanFilter2024}. 
Ensemble Kalman methods more generally can also be viewed in terms of measure transport~\cite{bergemann2010LocalizationTechniqueEnsemble,bergemann2010MollifiedEnsembleKalman,daum2010exact,gu2007IterativeEnsembleKalman,li2009IterativeEnsembleKalman,hertyKineticMethodsInverse2019,pidstrigach2023AffineInvariantEnsembleTransform,reich2015probabilistic,sakov2012iterative} with connections to sequential Monte Carlo methods~\cite{chopin2020introduction,delmoral2006SequentialMonteCarlo}; for further reading, see two excellent recent surveys~\cite{calvello2022ensemble,chadaIterativeEnsembleKalman2021}.

Our contributions in the present work are new spectral and convergence analyses of both the deterministic and stochastic versions of basic EKI (\Cref{alg:EKI}) for \textit{linear} observation operators $\bH$. 
Convergence results for some classes of nonlinear observers have previously been derived in~\cite{weissmann2022gradient,chada2022convergence,wagnerEnsembleKalmanFilter2022}; 
however, our focus on the linear case  yields a natural interpretation of EKI's convergence properties in terms of `fundamental subspaces of EKI', akin to the four `fundamental subspaces' that characterize Strang's `Fundamental Theorem of Linear Algebra'~\cite{strang1993fundamental}.
We directly analyze discrete EKI iterations rather than their continuous-time limits, and utilize spectral decompositions to define three fundamental subspaces in both the observation space $\R^n$, and the state space $\R^d$, yielding a total of \textit{six} fundamental subspaces of EKI. 
We directly consider the discrete iteration for two principal reasons: (i)~continuous-time limits can exhibit instabilities that are not observed for the discrete algorithms~\cite{armbrusterStabilizationContinuousLimit2022}, and (ii) we can state and prove results about the discrete iteration without involving the properties of solutions to ODEs and SDEs, making the linear theory of EKI more broadly accessible to a wider audience, particularly in the stochastic setting. Our specific contributions are:
\begin{enumerate}
    \item New spectral and convergence analysis of the \textit{discrete} deterministic EKI iteration in both the observation and state spaces that (a) verify convergence rates previously derived in the continuous-time limit, (b) unify previous results to define fundamental subspaces of EKI for the general case where $\bH$ and the ensemble covariance may both be low-rank, and (c)~yield new results relating EKI solutions to the standard minimum-norm least squares solution.
    \item New spectral and convergence analysis of the \textit{discrete} stochastic EKI iteration based on an idealized covariance iteration that reflects the expected statistics conditioned on noise at previous iterations. This yields (a) new results concerning the convergence of linear stochastic EKI particles in fundamental subspaces that are analogous to those of deterministic EKI, and (b) new insight into the failure of stochastic EKI to converge for small ensemble sizes.
    \item Numerical experiments demonstrating that the derived analytical results are descriptive.
\end{enumerate}

The remainder of this work is organized as follows.
\Cref{sec: fundamental subspaces} summarizes our main results. \Cref{sec: deterministic analysis,sec: stochastic analysis} prove results for deterministic and stochastic linear EKI, respectively. \Cref{sec: numerics} numerically illustrates our results and \Cref{sec: conclusions} concludes.

\section{The fundamental subspaces of Ensemble Kalman Inversion}\label{sec: fundamental subspaces}
For the remainder of this work, we assume $\bH$ is a linear operator, $\bH\in\R^{n\times d}$, with rank $h \leq \min(n,d)$. We begin by reviewing the standard four fundamental subspaces of the weighted least squares problem (\Cref{subsec: LS fundamental subspaces}). In \Cref{subsec: EKI fundamental subspaces}, we then introduce the six fundamental subspaces of EKI at a high level and summarize our convergence results. Detailed technical definitions and proofs are deferred to \Cref{sec: deterministic analysis,sec: stochastic analysis}.

\begin{figure}
    \centering
    \includegraphics[width=0.95\linewidth,trim={0 0.8cm 0 0},clip]{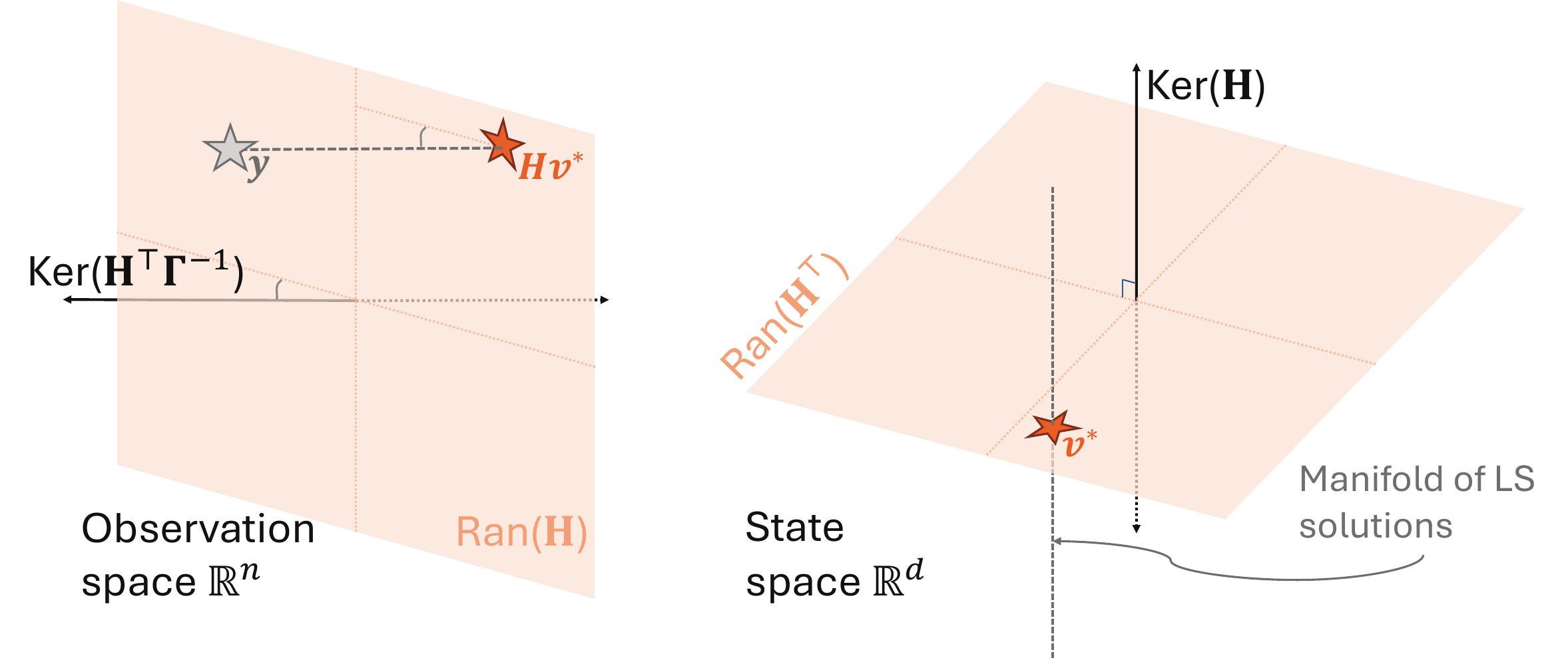}
    \caption{The four fundamental subspaces of the weighted least squares problem~\eqref{eq: least squares problem} for linear observers, $\bH\in\R^{n\times d}$. In observation space (left), the image of the least squares solution (orange star) is the $\bGamma^{-1}$-orthogonal projection of the data (gray star) onto $\Ran(\bH)$. In state space (right), the minimum-norm least squares solution (orange star) lies in $\Ran(\bH^\top)$ with zero component in $\Ker(\bH)$. }
    \label{fig: four LA fundamental subspaces}
\end{figure}

\subsection{Fundamental subspaces of weighted least squares problems}\label{subsec: LS fundamental subspaces}
In the context of the weighted least-squares problem~\eqref{eq: least squares problem}, the four fundamental subspaces arise from dividing the two fundamental \textit{spaces} (observation space $\R^n$ and state space $\R^d$) into two \textit{sub}spaces each, one subspace associated with `observable' directions and a complementary subspace associated with `unobservable' directions as depicted in \Cref{fig: four LA fundamental subspaces}. In observation space $\R^n$, the two fundamental subspaces are the range of $\bH$, denoted $\Ran(\bH)$, and its $\bGamma^{-1}$-orthogonal complement, $\Ker(\bH^\top\bGamma^{-1})$ (denoting the nullspace of $\bH^\top\bGamma^{-1}$). In our least-squares problem, the closest that $\bH\bv$ can come to $\by\in\R^n$ with respect to the $\bGamma^{-1}$-norm is the $\bGamma^{-1}$-orthogonal projection of $\by$ onto the observable space $\Ran(\bH)$, which then has a zero component in the (unobservable) subspace $\Ker(\bH^\top\bGamma^{-1})$.
In state space $\R^d$, the two fundamental subspaces are $\Ran(\bH^\top)$, and its orthogonal complement with respect to the Euclidean norm, $\Ker(\bH)$. State space directions in $\Ker(\bH)$ are unobservable because they are mapped by $\bH$ to zero and thus do not influence the minimand of~\eqref{eq: least squares problem}. If $\Ker(\bH)$ is non-trivial, multiple minimizers of~\eqref{eq: least squares problem} exist. 
 A standard choice making the problem~\eqref{eq: least squares problem} well-posed is a norm-minimizing solution given by:
\begin{equation}\label{eq: ls solution}
    \bv^* = (\bH^\top\bGamma^{-1}\bH)^{\dagger}\bH^\top\bGamma^{-1}\by\equiv \bH^+\by,
\end{equation}
where ``$\dagger$" denotes the usual Moore-Penrose pseudoinverse and we have introduced the weighted pseudoinverse, $\bH^+=(\bH^\top\bGamma^{-1}\bH)^{\dagger}\bH^\top\bGamma^{-1}$. 
This unique norm-minimizing solution to~\eqref{eq: least squares problem} lies in the observable space $\Ran(\bH^\top)$ and has a zero component in the unobservable space $\Ker(\bH)$. 

\subsection{Fundamental subspaces of Ensemble Kalman Inversion}\label{subsec: EKI fundamental subspaces}
Denote the empirical particle covariance by $\bC_i=\sfCov[\bv_i^{(1:J)}]$. 
In ensemble Kalman inversion, fundamental subspaces arise first from dividing the state and observation spaces into directions that are `populated' by particles, lying in the range of $\bC_i$ (it is well-known that $\Ran(\bC_i)$ is invariant for all $i$~\cite{blomker2019well,bungert2023complete,iglesias2013ensemble,schillings2017analysis}), and `unpopulated' directions lying in a complementary subspace. The populated subspace can then be further divided into two subspaces associated with observable and unobservable directions. There are therefore three subspaces in each of the observation and state spaces, and each of those sets of three subspaces are associated with three complementary oblique projection operators, $\bcP,\bcQ,\bcN\in\R^{n\times n}$ in observation space and $\bbP,\bbQ,\bbN\in\R^{d\times d}$ in state space, which we will  define precisely in subsequent sections. In observation space $\R^n$, the three fundamental subspaces are then (\textit{i}) $\Ran(\bcP)\equiv \Ran(\bH\bC_i)$, associated with observable populated directions, (\textit{ii}) $\Ran(\bcQ)\equiv \bH\,\Ker(\bC_i\fisher)$, associated with observable but \textit{un}populated directions, and (\textit{iii})~$\Ran(\bcN)\equiv\Ker(\bH^\top\bGamma^{-1})$, associated with unobservable directions. In state space $\R^d$, the three fundamental subspaces are (\textit{i}) $\Ran(\bbP)\subset\Ran(\bC_i)$, associated with observable populated directions (but generally \textit{not} simply the intersection of $\Ran(\bC_i)$ with $\Ran(\bH^\top)$); (\textit{ii})~$\Ran(\bbQ)\subset\Ran(\bH^\top)$ associated with observable \textit{un}populated directions, and (\textit{iii})~$\Ran(\bbN)$ associated with unobservable directions. In subsequent sections, we will describe more precise relationships between these subspaces and the operators $\bH$ and $\bC_i$. We will also show that EKI misfits [residuals] converge to zero at a $1/\sqrt{i}$ rate in the fundamental subspace associated with observable and populated directions, $\Ran(\bcP)$ [$\Ran(\bbP)$], and remain constant in the fundamental subspaces associated with observable unpopulated directions, $\Ran(\bcQ)$ [$\Ran(\bbQ)$], and likewise with unobservable directions, $\Ran(\bcN)$ [$\Ran(\bbN)$]. The fundamental subspaces of EKI are depicted in \Cref{fig: six EKI fundamental subspaces}, and an interactive three-dimensional visualization is available at \texttt{\url{https://elizqian.github.io/eki-fundamental-subspaces/}}.

\begin{figure}[t]
    \centering
    \includegraphics[width=0.95\linewidth]{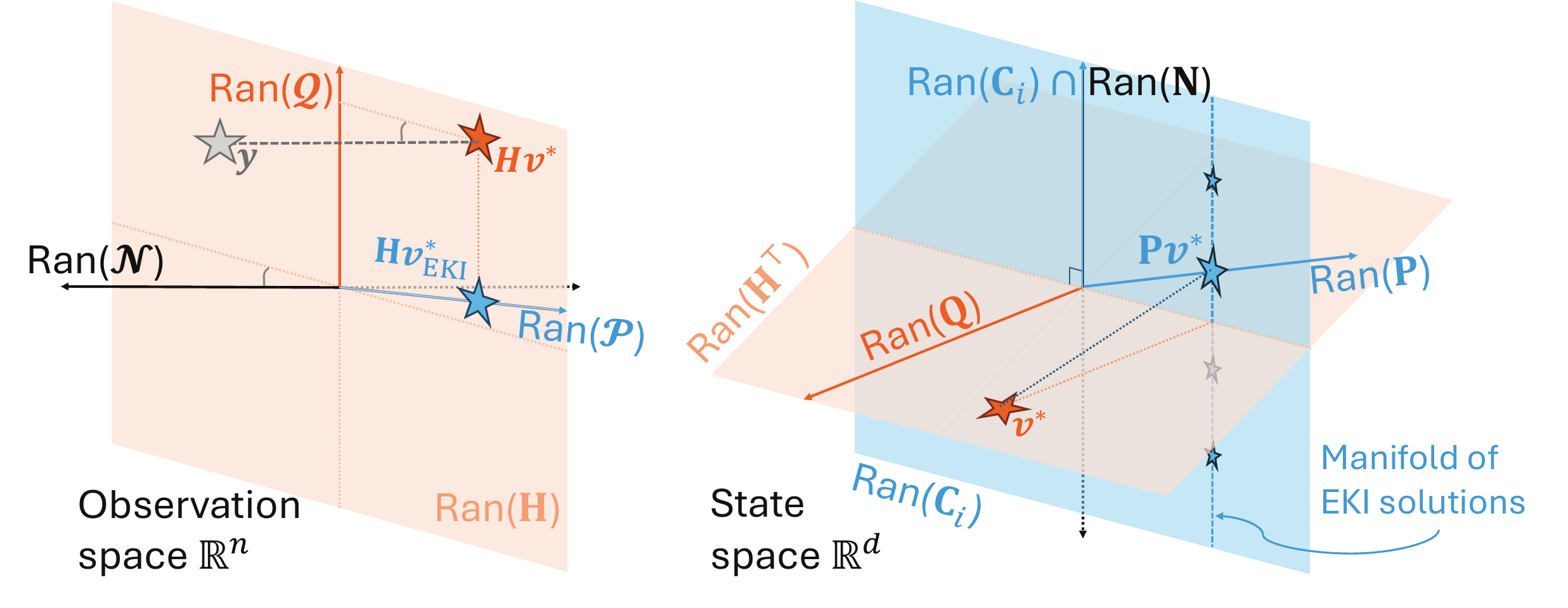}
    \caption{The six fundamental subspaces of Ensemble Kalman Inversion. In observation/state space (left/right), the three fundamental subspaces are (i) $\Ran(\bcP)$/$\Ran(\bbP)$, associated with observable populated directions, (ii) $\Ran(\bcQ)$/$\Ran(\bbQ)$ associated with observable unpopulated directions, and (iii)~$\Ran(\bcN)$/$\Ran(\bbN)$ associated with unobservable directions. In observation space, the image of the least squares solution (orange star) is projected by the oblique projector $\bcP$ to obtain the EKI solution (blue star). In state space, projection of the minimum norm solution (orange star) by the oblique projector $\bbP$ yields one possible EKI solution (large blue star) out of the EKI solution manifold (dashed blue line with small blue stars). }
    \label{fig: six EKI fundamental subspaces}
\end{figure}

\section{Analysis of deterministic EKI}\label{sec: deterministic analysis}
We first consider linear deterministic EKI (\Cref{alg:EKI} with the choice $\bSigma=\boldsymbol{0}$ so that $\by_i^{(j)} = \by$, and $\bH\in\R^{n\times d}$ a linear observer). 
\Cref{ssec: deterministic observation space,ssec: deterministic state space} develop results in observation space $\R^n$ and state space $\R^d$, respectively.

\subsection{Deterministic EKI: Analysis in observation space $\R^n$}\label{ssec: deterministic observation space}
We begin by deriving an iteration map that governs the evolution of the data misfit (\Cref{sssec: misfit iteration deterministic}). 
\Cref{sssec: observation spectral analysis deterministic} provides a spectral analysis of this iteration map that distinguishes three fundamental subspaces in observation space within which EKI has differing convergence behaviors, which are analyzed in \Cref{sssec: observation convergence deterministic}.

\subsubsection{Deterministic EKI: The data misfit iteration}\label{sssec: misfit iteration deterministic}
Let $\btheta_i^{(j)} = \bh_i^{(j)}-\by=\bH\bv_i^{(j)}-\by$, for $j=1,\ldots,J$ and $i=0,1,\ldots$, denote the \textit{data misfits}, representing the deviation of the image of the particles under $\bH$ from the true data.
Let $\bC_i=\sfCov[\bv_i^{(1:J)}]$ denote the empirical particle covariance. Note that for a linear observer, $\bH$, $\sfCov[\bv_i^{(1:J)},\bh_i^{(1:J)}]\equiv \bC_i\bH^\top$, $\sfCov[\bh_i^{(1:J)}]\equiv\bH\bC_i\bH^\top$, so that the Kalman gain is given as $\bK_i \equiv \bC_i\bH^\top(\bH\bC_i\bH^\top +\bGamma)^{-1}$. The update step of \Cref{alg:EKI} becomes:
\begin{align}\label{eq: EKI update}
    \bv_{i+1}^{(j)} &= \bv_i^{(j)}+\bC_i\bH^\top(\bH\bC_i\bH^\top +\bGamma)^{-1}(\by_i^{(j)}-\bH\bv_i^{(j)}).
\end{align}
Then, setting $\by_i^{(j)}=\by$ for deterministic EKI yields 
\begin{align}
    \btheta_{i+1}^{(j)} &= \bH\bv_{i+1}^{(j)}-\by = \bH\bv_i^{(j)}+\bH\bC_i\bH^\top(\bH\bC_i\bH^\top +\bGamma)^{-1}(\by-\bH\bv_i^{(j)})-\by\nonumber\\
    &= (\bI - \bH\bC_i\bH^\top(\bH\bC_i\bH^\top +\bGamma)^{-1})(\bH\bv_i^{(j)}-\by) \nonumber \\
    &= \bGamma (\bH\bC_i\bH^\top +\bGamma)^{-1}\btheta_i^{(j)}\equiv \bcalM_i\btheta_i^{(j)},
    \label{eq: deterministic misfit iteration}
\end{align}
where we have defined the data misfit iteration map, $\bcalM_i = \bGamma (\bH\bC_{i}\bH^\top + \bGamma)^{-1}$. 
Describing the evolution of the data misfit, $\btheta_i^{(j)}$, will require an analysis of $\bcalM_i$. 

\subsubsection{Deterministic EKI: Decomposition of observation space $\R^n$}\label{sssec: observation spectral analysis deterministic}
A spectral analysis of $\bcalM_i$ will distinguish three fundamental subspaces of the observation space $\R^n$ that are invariant under $\bcalM_i$.
Consider the generalized eigenvalue problem defined by the pencil $(\bH\bC_i\bH^\top,\bGamma)$: that is, let $(\delta_{\ell,i},\bw_{\ell,i})$ satisfy
\begin{align}
    \bH\bC_{i}\bH^\top\bw_{\ell, i} = \delta_{\ell,i} \bGamma\bw_{\ell, i}\quad \mbox{for}\quad \ell=1,\ldots,n\  \mbox{and}\  i\geq 0.
    \label{eq: observation GEV}
\end{align}
Note that all eigenvalues of~\eqref{eq: observation GEV} are non-negative. 
While the eigenvalues, $\delta_{\ell,i}$, will evolve with $i$, we find that the corresponding eigenvectors remain constant:
\begin{proposition}\label{prop: observation GEV behavior}
    The generalized eigenvectors of~\eqref{eq: observation GEV} are constant with respect to the iteration index, $i$, so that $\bw_{\ell,i+1} = \bw_{\ell,i} = \bw_\ell$. The corresponding eigenvalues evolve with respect to $i$ according to: $\delta_{\ell,i+1} = \delta_{\ell,i}/(1+\delta_{\ell,i})^2$.
\end{proposition}

\begin{proof}
Noting that $\btheta_i^{(j)}=\bh_i^{(j)}-\by$, and that (empirical) covariances are unchanged by translation yields $\sfCov[\btheta_i^{(1:J)},\btheta_i^{(1:J)}]=\sfCov[\bh_i^{(1:J)},\bh_i^{(1:J)}]=\bH\bC_i\bH^\top$. Then, 
\begin{align*}
    \bH\bC_{i+1}\bH^\top &= \sfCov[\btheta_{i+1}^{(1:J)},\btheta_{i+1}^{(1:J)}]=\bcalM_i\sfCov[\btheta_i^{(1:J)},\btheta_i^{(1:J)}]\bcalM_i^\top=\bcalM_i\bH\bC_i\bH^\top\bcalM_i^\top,\\
    &=\bGamma(\bH\bC_{i}\bH^\top+\bGamma)^{-1}\bH\bC_{i}\bH^\top(\bH\bC_{i}\bH^\top+\bGamma)^{-1}\bGamma.
\end{align*}
Suppose $(\delta_i,\bw_i)$ is an eigenpair of \eqref{eq: observation GEV} at iteration $i$, $\bH\bC_{i}\bH^\top\bw_i = \delta_i\bGamma\bw_i$. Then
\begin{align*}
    (\bH\bC_{i}\bH^\top+\bGamma)\bw_{i} = (1+\delta_i)\bGamma\bw_{i}
    \quad    \implies \quad
    (\bH\bC_{i}\bH^\top+\bGamma)^{-1}\bGamma\bw_{i} = \frac1{1+\delta_i}\bw_{i} \quad (*)
\end{align*}
which implies
\begin{align*}
    \bH\bC_i\bH^\top(\bH\bC_i\bH^\top + \bGamma)^{-1}\bGamma\bw_i = \bH\bC_i\bH^\top\frac1{1+\delta_i}\bw_i = \frac{\delta_i}{1+\delta_i}\bGamma\bw_i.
\end{align*}
Left-multiplying by $\bcalM_i = \bGamma(\bH\bC_i\bH^\top+\bGamma)^{-1}$ and applying ($*$) yields 
\begin{align*}
    \bH\bC_{i+1}\bH^\top\bw_i=\frac{\delta_i}{(1+\delta_i)^2}\bGamma\bw_i.
\end{align*}
\end{proof}

 We construct a basis for the observation space $\R^n$ made up of eigenvectors of~\eqref{eq: observation GEV}:
\begin{proposition}\label{prop: eigenvector basis deterministic observation}
    Zero eigenvalues of~\eqref{eq: observation GEV} remain zero for all iterations, i.e., $\delta_{\ell,0}=0$ implies $\delta_{\ell,i}=0$ for all subsequent $i\geq 1$. Similarly, $\delta_{\ell,0}>0$  implies $\delta_{\ell,i}>0$ for all $i\geq 1$. 
    Let $r$ denote the number of positive eigenvalues of~\eqref{eq: observation GEV} and let $h=\rank(\bH)$.
    Then, there is a $\bGamma$-orthogonal basis for $\R^n$ made up of eigenvectors of~\eqref{eq: observation GEV}, $\{\bw_1,\ldots,\bw_n\}$, such that
    \begin{enumerate}
        \item $\{\bw_1,\ldots,\bw_r\}\subset\Ran(\bGamma^{-1}\bH)$ are eigenvectors of~\eqref{eq: observation GEV} associated with positive eigenvalues $\delta_{1,i},\ldots,\delta_{r,i}$, labeled in non-increasing order at $i=1$, that is, $\delta_{1,1}\geq \delta_{2,1}\geq \cdots\geq\delta_{r,1}>0$. This ordering is preserved for subsequent $i\geq 1$.
        \item if $r<h$, $\{\bw_{r+1},\ldots,\bw_h\}\subset\Ran(\bGamma^{-1}\bH)$ are eigenvectors of~\eqref{eq: observation GEV} associated with zero eigenvalues, and
        \item if $h<n$, $\{\bw_{h+1},\ldots,\bw_n\}\subset\Ker(\bH^\top)$ are eigenvectors of~\eqref{eq: observation GEV} also associated with zero eigenvalues.
    \end{enumerate}
    Additionally, $\textsf{span}(\bw_1,\ldots,\bw_h)=\Ran(\bGamma^{-1}\bH)$ and $\textsf{span}(\bw_{h+1},\ldots,\bw_n)=\Ker(\bH^\top)$.
\end{proposition}

\begin{proof}
The preservation of eigenvalue (non-)zeroness can be read directly from the eigenvalue recurrence in \Cref{prop: observation GEV behavior}.
We now detail the construction of an eigenvector basis satisfying the above conditions.
Take $\bw_1,\ldots,\bw_r$ to be the linearly independent eigenvectors associated with the $r$ positive eigenvalues of \eqref{eq: observation GEV} ordered so that $\delta_{1,1}\geq \delta_{2,1}\geq \cdots\geq\delta_{r,1}>0$.
Note that the recurrence relation in \Cref{prop: observation GEV behavior} implies that all nonzero eigenvalues must be in the interval $(0,1)$  for all $i\geq1$. Preservation of ordering follows from the fact that $\frac{x}{(1+x)^2}$ is monotone increasing for $x\in(0,1)$. 
Note that $\{\bw_1,\ldots,\bw_r\}$ are $\bGamma$-orthogonal and that $\textsf{span}(\bw_1,\ldots,\bw_r)=\Ran(\bGamma^{-1}\bH\bC_i\bH^\top)\subset\Ran(\bGamma^{-1}\bH)$.
  
If $r<h$, complete a $\bGamma$-orthogonal basis for $\Ran(\bGamma^{-1}\bH)$ with an additional $(h-r)$ linearly independent vectors $\{\bw_{r+1},\ldots,\bw_h\}$ so that $\textsf{span}(\bw_1,\ldots,\bw_h)=\Ran(\bGamma^{-1}\bH)$. 
Suppose $\widehat{\mathbf{w}}\in \mathsf{span}(\bw_{r+1},\ldots,\bw_h)$ is chosen arbitrarily. 
Since $\{\bw_1,\ldots,\bw_r\}$ form a basis for $\Ran(\bGamma^{-1}\bH\bC_i\bH^\top)$, $\widehat\bw$ must be $\bGamma$-orthogonal to $\Ran(\bGamma^{-1}\bH\bC_i\bH^\top)$. In particular, $\widehat\bw^\top\bGamma\bGamma^{-1}\bH\bC_i\bH^\top=0$ which implies $\bH\bC_i\bH^\top\widehat\bw=0$, so each of
$\{\bw_{r+1},\ldots,\bw_h\}$ are eigenvectors of \eqref{eq: observation GEV} associated with the zero eigenvalue. 

Finally, if $h<n$, choose  $(n-h)$ vectors, labeled as $\{\bw_{h+1},\ldots,\bw_n\}$, to be a $\bGamma$-orthogonal basis for $\Ker(\bH^\top)$. They then will be eigenvectors for \eqref{eq: observation GEV} associated with the zero eigenvalue. The full set of vectors, 
$\{\bw_1,\ldots,\bw_r,\bw_{r+1},\ldots,\bw_h,\bw_{h+1},\ldots,\bw_n\}$
will be a $\bGamma$-orthogonal basis of eigenvectors of \eqref{eq: observation GEV} that span $\mathbb{R}^{n}$. 
\end{proof}

\begin{remark}
    Note that $r=\rank(\bH\bC_i\bH^\top)$ since we have assumed $\bGamma$ is invertible. In fact, $r$ may be more generally thought of as the `rank' of a particular instance of EKI, because it is the dimension of the subspace in which EKI converges, as our further analysis will show. Note that $r\leq\min(h,\rank(\bC_i))$. 
\end{remark}

This construction leads us to a compact summary representation of \eqref{eq: observation GEV}: define $\bDelta_i = {\sf diag}(\delta_{1,i},\delta_{2,i},\ldots,\delta_{n,i})$ and  $\bW=[\bw_1,\bw_2,\ldots,\bw_n]$.
Then \eqref{eq: observation GEV} can be written as $\bH\bC_i\bH^\top\bW =  \bGamma\bW\bDelta_i$. $\bW$ is the matrix of eigenvectors of \eqref{eq: observation GEV} and $\bGamma$-orthogonality of the eigenvectors establishes that $\bW^\top\bGamma\bW$ is diagonal. 
Without loss of generality we may renormalize eigenvectors so that $\bw_\ell^\top\bGamma\bw_\ell=1$, or equivalently,  $\bW^\top\bGamma\bW=\bI$.   

The three types of eigenvectors of~\eqref{eq: observation GEV} described in \Cref{prop: eigenvector basis deterministic observation} are each associated with different invariant subspaces of the data misfit iteration map $\bcalM_i$, which we now characterize in terms of spectral projectors of $\bcalM_i$.

\begin{proposition}\label{def: observation projectors}
    Let  $\bW_{k:\ell}\in\R^{n\times(\ell-k+1)}$ contain columns $k$ through $\ell$ of $\bW$. Define
     $ \bcP = \bGamma\bW_{1:r}\bW_{1:r}^\top$,
        $\bcQ = \bGamma\bW_{r+1:h}\bW_{r+1:h}^\top$,  and %\quad 
        $\bcN = \bGamma\bW_{h+1:n}\bW_{h+1:n}^\top. $
    Then, $\bcP$, $\bcQ$, and $\bcN$ are  spectral projectors associated with the data misfit iteration map $\bcalM_i$, i.e., 
        $\bcalM_i\bcP = \bcP\bcalM_i$ 
        and $\bcP^2 = \bcP$,
    with similar assertions for $\bcQ$ and $\bcN$.  $\bcP$, $\bcQ$, and $\bcN$ are complementary in the sense that  $\bcP\bcQ = \bcP\bcN = \bcQ\bcN =\boldsymbol{0}$ and $\bcP+\bcQ + \bcN = \bI_n$.
\end{proposition}

\begin{proof}
Note $\bcalM_i = \bGamma\bW(\bI_n+\bDelta_i)^{-1}\bW^\top$. Define $\bDelta_{1:r}^{(i)} = \textsf{diag}(\delta_{1,i},\ldots,\delta_{r,i})$. Then, 
\begin{align*}
    \bcalM_i\bcP &= \bGamma\bW(\bI_n+\bDelta_i)^{-1}\bW^\top \bGamma\bW_{1:r}\bW_{1:r}^\top = \bGamma\bW_{1:r}(\bI_r+\bDelta_{1:r}^{(i)})^{-1}\bW_{1:r}^\top\\
    &=\bGamma\bW_{1:r}\bW_{1:r}^\top\bGamma\bW(\bI_n+\bDelta_i)^{-1}\bW^\top = \bcP\bcalM_i.
\end{align*}
The normalization condition implies $\bW_{1:r}^\top \bGamma\bW_{1:r}=\bI_r$, so we have upon substitution $\bcP^2 =\bGamma\;\bW_{1:r}(\bW_{1:r}^\top\bGamma\;\bW_{1:r})\bW_{1:r}^\top =\bcP$.
Similar arguments can be followed to show that $\bcalM_i\,\bcQ = \bcQ\bcalM_i$ and  $\bcQ^2 = \bcQ$, as well as $\bcalM_i\,\bcN = \bcN\bcalM_i$ and  $\bcN^{\,2} = \bcN$.  
 The assertion $\bcP\bcQ = \bcP\bcN = \bcQ\bcN =\boldsymbol{0}$ follows immediately from the $\bGamma$-orthogonality of the eigenvector basis.  Likewise,  $\bW^\top\bGamma\bW=\bI$ implies $\bGamma^{-1}=\bW\bW^\top=\bW_{1:r}\bW_{1:r}^\top+\bW_{r+1:h}\bW_{r+1:h}^\top+\bW_{h+1:n}\bW_{h+1:n}^\top$.  Multiplication by $\bGamma$ verifies $\bcP+\bcQ + \bcN = \bI$. 
\end{proof}
We briefly explain our notational choices for the projectors $\bcP,\bcQ,\bcN$: we use $\bcN$ to denote the projector associated with the unobservable directions that lie in the kernel (\textit{n}ullspace) of $\bH^\top$. In the remaining observable directions, we use $\bcP$ to denote the projector associated with directions that are \textit{p}opulated by the particle covariance, and use $\bcQ$ to denote the complementary projector to $\bcP$ within the observable subspace.
The data misfits can thus be divided into three components associated with the oblique projectors defined in \Cref{def: observation projectors}: $\btheta_i^{(j)} = \bcP\btheta_i^{(j)}+\bcQ\btheta_i^{(j)}+\bcN\btheta_i^{(j)}$. 
These components will \emph{not} generally be orthogonal to one another, though they each occupy invariant subspaces that reflect differing convergence behaviors as we show in the next section.

\begin{remark}\label{rem: trivial obs space}
    Note that it is possible for any of our three subspaces to become trivial in special cases: if $\Ran(\bH)=\R^n$ then $\Ran(\bcN)$ is trivially $\{\mathbf{0}\}$ because all directions are observable. If $\bC_i$ is full-rank then $\Ran(\bcQ)$ is trivially $\{\mathbf{0}\}$. If $\Ran(\bC_i)\subset \Ker(\bH)$ then that would mean the ensemble would be completely unobservable and thus  $\Ran(\bcP)$ is trivially $\{\mathbf{0}\}$.
\end{remark}

\subsubsection{Deterministic EKI: Convergence analysis in observation space $\R^n$}\label{sssec: observation convergence deterministic}
We first prove a lemma concerning the reciprocal nonzero eigenvalues.
\begin{lemma}\label{lem: reciprocal recurrence}
    If $\delta_{\ell,0}\neq0$, then for all $i\geq1$, we have $\frac1{\delta_{\ell,i}} = \frac1{\delta_{\ell,0}}+ 2i + \sum_{k=0}^{i-1}\delta_{\ell,k}$. 
\end{lemma}
\begin{proof}
    We prove by induction, verifying first  the base case with $i=1$: 
    $$
\frac1{\delta_{\ell,1}}=\frac{(1+\delta_{\ell,0})^2}{\delta_{\ell,0}} = \frac{1+2\delta_{\ell,0} + \delta_{\ell,0}^2}{\delta_{\ell,0}} = \frac1{\delta_{\ell,0}} + 2 + \delta_{\ell,0}.
$$ 
The same calculation for $i\geq 1$ yields $\frac{1}{\delta_{\ell,i+1}}=\frac1{\delta_{\ell,i}} + 2 +\delta_{\ell,i}$. Substituting the inductive hypothesis for $\frac1{\delta_{\ell,i}}$ and re-arranging completes the induction step and the proof.
\end{proof}

We use this lemma to provide upper and lower bounds on the asymptotic behavior of the nonzero eigenvalues. We make use of the following fact:
\begin{fact}[\cite{young1991EulerConstant}]\label{fact: reciprocal sum bounds}
    For all $i\geq 1$,  $\log i <\sum_{k=1}^i\frac1k < 2 + \log i$. 
\end{fact}

\begin{proposition}\label{prop: deterministic eigenvalue bounds}
    If $\delta_{\ell,0}\neq0$, then there exists a constant $c_\ell>0$ so that for all $i\geq1$ the following bounds hold:
    \begin{align*}
        \frac1{2i}-\frac{c_\ell+\log i}{8i^2} < \delta_{\ell,i} < \frac1{2i}.
    \end{align*}
\end{proposition}
\begin{proof}
    From \Cref{lem: reciprocal recurrence} we have $\frac1{\delta_{\ell,i}}>2i$ for $i\geq1$ and therefore $\delta_{\ell,i}<\frac1{2i}$. Then, we can bound the sum $\sum_{k=0}^{i-1}\delta_{\ell,k}\leq \delta_{\ell,0} + \sum_{k=1}^{i-1}\frac1{2k}$ and apply \Cref{fact: reciprocal sum bounds} to obtain
    \begin{align}\label{eq: log i sum bound}
        \sum_{k=0}^{i-1}\delta_{\ell,k} \leq \delta_{\ell,0} + 1 + \frac12\log i. \tag{\#}
    \end{align}
    Using this bound in the recurrence from \Cref{lem: reciprocal recurrence} yields
    \begin{align*}
        \frac1{\delta_{\ell,i}}  <  2i + ({\delta_{\ell,0}}^{-1} + \delta_{\ell,0} + 1) + \frac12\log i.
    \end{align*}
Let $c_\ell/2 = ({\delta_{\ell,0}}^{-1} + \delta_{\ell,0} + 1)$, so that $c_\ell$ is a constant independent of $i$.
    Then,
    \begin{align*}
        \delta_{\ell,i}>\frac1{2i+\frac{c_\ell + \log i}{2}} = \frac1{2i}\left(\frac{1}{1+\frac{c_\ell + \log i}{4i}}\right) \geq \frac1{2i}\left(1-\frac{c_\ell+\log i}{4i}\right)= \frac1{2i}-\frac{c_\ell+\log i}{8i^2}.
    \end{align*}
\end{proof}

\begin{corollary}
    As $i\to\infty$, the nonzero eigenvalues of~\eqref{eq: observation GEV} satisfy $\delta_{\ell,i}\sim \frac1{2i}$.
\end{corollary}
\begin{proof}
    From~\Cref{lem: reciprocal recurrence},
    \begin{align*}
        \lim_{i\to\infty} \left(\frac1{2i}\bigg/\delta_{\ell,i}\right) = \lim_{i\to\infty} \frac1{2i \delta_{\ell,0}} + 1 + \lim_{i\to\infty}\frac1{2i}\sum_{k=0}^{i-1}\delta_{\ell,k}.
    \end{align*}
    The first limit on the right hand side is evidently zero. Using \eqref{eq: log i sum bound}, the third term on the right hand side can be bounded as $0<\frac1{2i}\sum_{k=0}^{i-1}\delta_{\ell,k} < \frac1{2i}(\delta_{\ell,0}+1+\frac12\log{i})\rightarrow 0$,
    as $i\to\infty$.  Thus, $\lim_{i\to\infty}\frac1{2i} / \delta_{\ell,i}=1$.
\end{proof}

We can now prove our main result about the behavior of the EKI data misfit in the observation space $\R^n$. We use $\|\cdot\|$ to denote the Euclidean norm.

\begin{theorem}\label{thm: deterministic observation particles}
    For all particles $j=1,2,\ldots,J$, the following hold: 
    \begin{enumerate}[label=(\alph*)]
        \item as $i\to\infty$, $\|\bcP\btheta_i^{(j)}\|=\mathcal{O}(i^{-\frac12})$, 
        \item for all $i\geq0$, $\bcQ\btheta_i^{(j)}=\bcQ\btheta_0^{(j)}$, and
        \item for all $i\geq0$,  $\bcN\btheta_i^{(j)}=\bcN\btheta_0^{(j)}$.
    \end{enumerate}
\end{theorem}
\begin{proof}
    Note that $\bcalM_i = \bGamma\bW_{1:r}(\bI_r+\bDelta_{1:r}^{(i)})^{-1}\bW_{1:r}^\top +\bcQ + \bcN$. 
    Define $\bcalM_{i0}=\bcalM_i\cdots\bcalM_2\bcalM_1\bcalM_0$ so that $\btheta_i^{(j)} =\bcalM_{i0}\btheta_0^{(j)}$, and note that
    \begin{align*}
        \bcalM_{i0} =\bGamma\bW_{1:r}\bD_{i0}\bW_{1:r}^\top + \bcQ + \bcN,
        \quad \text{where}\quad
        \bD_{i0} = \prod_{k=0}^i(\bI_r+\bDelta_{1:r}^{(k)})^{-1}.
    \end{align*}
    Then, for all $i$, $\bcQ\bcalM_{i0} = \bcQ$ and $\bcN\bcalM_{i0}=\bcN$, and the last two statements of \Cref{thm: deterministic observation particles} follow.
    Additionally, 
    \begin{align*}
        \bcP\btheta_i^{(j)} &=\bcP\bcalM_{i0}\btheta_0^{(j)} 
        = \bcalM_{i0}\bcP\btheta_0^{(j)}=\bGamma\bW_{1:r}\bD_{i0}\bW_{1:r}^\top\bcP\btheta_0^{(j)}.
    \end{align*}
    Denote the diagonal entries of $\bD_{i0}$ as $d_{\ell,i}=\prod_{k=0}^i(1+\delta_{\ell,k})^{-1}$ for $1\leq\ell\leq r$. Then, 
    \begin{align*}
        \|\bcP\btheta_i^{(j)}\|\leq d_{r,i}\|\bGamma\bW_{1:r}\|\|\bW_{1:r}\|\|\bcP\btheta_0^{(j)}\|.
    \end{align*}
    For each particle $j$, the quantity $\|\bGamma\bW_{1:r}\|\|\bW_{1:r}\|\|\bcP\btheta_0^{(j)}\|$ is a constant independent of $i$, so we now need only show that $d_{r,i}=\mathcal{O}(i^{-\frac12})$.
    Note from the definition that $-\log d_{r,i}=\sum_{k=0}^i\log(1+\delta_{r,k})$.
    Since $\log(1+x)>x-\frac12x^2$ for $x\in(0,1)$, we have
    \begin{align*}
        -\log d_{r,i}> \log(1+\delta_{r,0}) + \sum_{k=1}^i \left(\delta_{r,k} - \frac12\delta_{r,k}^2\right).
    \end{align*}
    Moreover, since $x-\frac12x^2$ is monotone increasing on $(0,1)$,  \Cref{prop: deterministic eigenvalue bounds} implies
    \begin{align*}
        -\log d_{r,i}&> \log(1+\delta_{r,0}) + \sum_{k=1}^i \left(\frac1{2k}-\frac{c_r + \log k}{8k^2} - \frac12\left(\frac1{2k}-\frac{c_r + \log k}{8k^2}\right)^2 \right),\\
        &>\log(1+\delta_{r,0}) +\sum_{k=1}^i\frac1{2k} - \sum_{k=1}^\infty \left(\frac{c_r + \log k}{8k^2} + \frac12\left(\frac1{2k}-\frac{c_r + \log k}{8k^2}\right)^2\right)
    \end{align*}
    Note that all terms of the final summand converge at least as fast as $\frac{\log k}{k^2}$, so the final sum must converge to some finite constant $c_\infty$. Define $C=\log(1+\delta_{r,0})-c_\infty$.  
    Combining with the lower bound on the harmonic sum $\sum_{k=1}^i\frac1{2k}$ from \Cref{fact: reciprocal sum bounds}, yields
\begin{align*}-\log d_{r,i} > C + \frac12\log i\implies 
        d_{r,i} < \frac{\exp(-C)}{\sqrt{i}}=\frac{\exp(c_\infty)}{1+\delta_{r,0}}\frac{1}{\sqrt{i}}.
\end{align*}
\end{proof}

\begin{corollary}\label{cor: obs particle convergence deterministic}
    For each particle $j  = 1,2,\ldots, J$, the image of the particle under $\bH$, $\bh_i^{(j)}\equiv\bH\bv_i^{(j)}$, satisfies
    \begin{align*}
        \lim_{i\to\infty}\bh_i^{(j)}= \bcP\by + \bcQ\bh_0^{(j)} + \bcN\bh_0^{(j)}.
    \end{align*}
\end{corollary}
Thus, the observation-space particle converges to $\by$ in its $\bcP$-component, while its $\bcQ$- and $\bcN$-components remain at their initial values.

\paragraph{Related work}
Previous analyses of linear deterministic EKI have divided the misfit into just two components, corresponding to our $\bcP$ and its complementary projection under the assumption that $\bH$ is one-to-one~\cite{schillings2017analysis,schillings2018convergence}, or to our $\bcP$ and $\bcN$  under the assumption that $\bC_i$ is full-rank~\cite{bungert2023complete}.
\Cref{thm: deterministic observation particles} unifies these results in the general case where both $\bH$ and $\bC_i$ may be rank-deficient, leading to the definition of three fundamental subspaces: $\Ran(\bcP)$, comprised of populated observable directions, in which the data misfit converges at a $1/\sqrt{i}$ rate; and two subspaces in which the data misfit remains constant: directions in $\Ran(\bcQ)$ are observable but unpopulated by the ensemble, while directions in $\Ran(\bcN)$ are simply unobservable.

\subsection{Deterministic EKI: Analysis in state space $\R^d$}\label{ssec: deterministic state space}
We now consider the evolution of the state-space least squares residual of the particles with respect to the standard minimum-norm solution~\eqref{eq: ls solution}. \Cref{sssec: deterministic LS residual iteration} derives an iteration map for this residual, which motivates a spectral analysis in \Cref{sssec: deterministic spectral state space} that defines three fundamental subspaces of EKI in state space analogous to those defined in observation space. \Cref{sssec: deterministic state convergence} provides a convergence analysis that describes the behavior of the EKI particles in the three fundamental subspaces.

\subsubsection{Deterministic EKI: The least-squares residual iteration}\label{sssec: deterministic LS residual iteration}
We first show the well-known `subspace property' of basic EKI~\cite{blomker2019well,bungert2023complete,iglesias2013ensemble,schillings2017analysis}:
that is, across all iterations particles always remain in the subspace spanned by the initial ensemble.
\begin{proposition}\label{prop: subspace property}
For all $i\geq 0$, $\Ran(\bC_{i+1})\subset \Ran(\bC_i)$. 
Additionally, define $\mathbf{V}_i=[\bv_i^{(1)},\bv_i^{(2)},\ldots,\bv_i^{(J)}]\in\mathbb{R}^{d\times J}$. Then $\Ran(\bV_{i+1})\subset\Ran(\bV_i)\subset\Ran(\bV_0)$.
\end{proposition}
\begin{proof}
Define $\mathbf{e}=[1,1,\ldots,1]^\top\in\mathbb{R}^J$ and $\boldsymbol{\Pi}=\frac1J\mathbf{e}\mathbf{e}^\top$.  Then $\boldsymbol{\Pi}$ is an orthogonal projection, $\mathbf{V}_i\boldsymbol{\Pi}=\sfE[\bv_{i}^{(1:J)}]$, and $\bC_i=\frac{1}{J-1}\mathbf{V}_i(\bI-\boldsymbol{\Pi})\mathbf{V}_i^\top$, so that $\Ran(\bC_i)=\Ran(\mathbf{V}_i(\bI-\boldsymbol{\Pi}))$.  Write $\mathbf{Y}_i=[\by_i^{(1)},\by_i^{(2)},\ldots,\by_i^{(J)}]\in\mathbb{R}^{d\times J}$ and express  the EKI update~\eqref{eq: EKI update} as
$\mathbf{V}_{i+1} = \mathbf{V}_i+\bC_i\bH^\top(\bH\bC_i\bH^\top +\bGamma)^{-1}(\mathbf{Y}_i-\bH\mathbf{V}_i).$
Postmultiplying by $(\bI-\boldsymbol{\Pi})$ 
leads to the first conclusion:  
\begin{align*}
\Ran(\bC_{i+1})=\Ran(\mathbf{V}_{i+1}(\bI-\boldsymbol{\Pi}))\subset \Ran(\mathbf{V}_{i}(\bI-\boldsymbol{\Pi})) =\Ran(\bC_i)\subset \Ran(\bV_i). 
\end{align*}
Postmultiplying instead by $\boldsymbol{\Pi}$ 
leads to $\sfE[\bv_{i+1}^{(1:J)}]=\mathbf{V}_{i+1}\boldsymbol{\Pi}\in\Ran(\mathbf{V}_{i})$, so we have 
\begin{align*}
\bv_{i+1}^{(j)}\in\sfE[\bv_{i}^{(1:J)}]+\Ran(\bC_i)\subset\Ran(\mathbf{V}_{i}),
\end{align*}
which yields the second conclusion.
\end{proof}

\begin{remark}
    Note that \Cref{prop: subspace property} holds for both the deterministic and stochastic versions of basic EKI (\Cref{alg:EKI}), although the remainder of this section focuses solely on the deterministic case.
\end{remark}

Denote by $\bomega_i^{(j)}=\bv_i^{(j)}-\bv^*$ the residual between the $j$th particle and the minimum-norm least squares solution~\eqref{eq: ls solution}. We now show an evolution map for this residual. 

\begin{proposition} \label{prop_statespace_convergence}
    For $i\geq 0$, $\bomega_{i+1}^{(j)} = \bbM_i\bomega_i^{(j)}$ with $\bbM_i=(\bI+\bC_i\bH^\top\bGamma^{-1}\bH)^{-1}$.
\end{proposition}
\begin{proof}
Let $\br = \bH\bv^*-\by$ denote the least squares misfit and note that $\br\in\Ker(\bH^\top\bGamma^{-1})$. Then, subtracting $\bv^*$ from the EKI update yields
\begin{align*}
    \bomega_{i+1}^{(j)} = \bv_{i+1}^{(j)} - \bv^* = (\bI-\bK_i\bH)\bv_i^{(j)}-\bv^*+\bK_i\by = (\bI-\bK_i\bH)(\bv_i^{(j)} - \bv^*) - \bK_i\br.
\end{align*}
We show $\bK_i\br=\mathbf{0}$ and $(\bI-\bK_i\bH)=\bbM_i$. First, observe that
\begin{align*}
    (\bI-\bK_i\bH)\bC_i\bH^\top\bGamma^{-1} &= \bC_i\bH^\top\bGamma^{-1} -\bC_i\bH^\top(\bH\bC_i\bH^\top + \bGamma)^{-1} \bH\bC_i\bH^\top\bGamma^{-1}\\
    &= \bC_i\bH^\top (\bI - (\bH\bC_i\bH^\top + \bGamma)^{-1} \bH\bC_i\bH^\top)\bGamma^{-1}\\
    &=\bC_i\bH^\top(\bH\bC_i\bH^\top + \bGamma)^{-1} = \bK_i,
\end{align*}
which makes it evident that $\Ker(\bH^\top\bGamma^{-1})\subset \Ker(\bK_i)$, so $\bK_i\br=\mathbf{0}$. Finally note that the above calculation implies $\bK_i\bH(\bI+\bC_i\bH^\top\bGamma^{-1}\bH) =\bC_i\bH^\top\bGamma^{-1}\bH$ which we can rearrange to $\bK_i\bH=\bC_i\bH^\top\bGamma^{-1}\bH\bbM_i=\mathbf{I}-\bbM_i$, so $(\bI-\bK_i\bH)=\bbM_i$.
\end{proof}

\subsubsection{Deterministic EKI: Decomposition of state space $\R^d$}
\label{sssec: deterministic spectral state space}
A spectral analysis of $\bbM_i$ will distinguish three fundamental subspaces of the state space $\R^d$ that are invariant under $\bbM_i$. Consider the following eigenvalue problem:
\begin{equation}
\bC_{i}\,\bH^\top\,\bGamma^{-1}\,\bH\,\bu_{\ell, i}=\delta'_{\ell, i}\bu_{\ell ,i}, \quad \ell =1,\ldots,d. \label{eq:GEVstatespace}
\end{equation}
The state space eigenpairs $(\delta'_{\ell, i},\bu_{\ell ,i})$ of \eqref{eq:GEVstatespace} will be seen to be closely related to the observation space eigenpairs $(\delta_{\ell, i},\bw_\ell)$ of~\eqref{eq: observation GEV}: 

\begin{proposition}\label{prop: deterministic state eigen basis}
    Let $\{\bw_\ell\}_{\ell=1}^n$ denote eigenvectors of~\eqref{eq: observation GEV} ordered as described in \Cref{prop: eigenvector basis deterministic observation}, and recall that the leading $r$ eigenvectors correspond to positive eigenvalues $\delta_{\ell,i}$. For $\ell=1,\ldots,r$, define $\bu_\ell=\frac1{\delta_{\ell,i}}\bC_i\bH^\top\bw_\ell$. 
    If $h\equiv\rank(\bH)>r$, then for $\ell=r+1,\ldots,h$, define $\bu_\ell=\bH^+\bGamma\bw_\ell$.  Then for all $\ell\leq h$, $\bu_\ell$ is an eigenvector of~\eqref{eq:GEVstatespace} with eigenvalue $\delta'_{\ell, i}=\delta_{\ell,i}$. Conversely, for all $\ell\leq h$, $\bw_\ell = \bGamma^{-1}\bH\bu_\ell$.
\end{proposition}
\begin{proof}
    One may verify the relationships between $\bw_\ell$ and $\bu_\ell$ by direct substitution, recalling that $\bH\bH^+$ is the $\bGamma^{-1}$-orthogonal projection onto $\Ran(\bH)$ and that for $\ell\leq h$, $\bw_\ell\in\Ran(\bGamma^{-1}\bH)$ or equivalently,  $\bGamma\bw_\ell\in\Ran(\bH)$.
    We need still to confirm that the eigenvectors $\bu_\ell$ are constant with respect to the iteration index,  $i$. Recalling $\bbM_i = (\bI-\bK_i\bH)$ and assuming $(\delta_{\ell,i}',\bu_\ell)$ satisfies~\eqref{eq:GEVstatespace}, we have
    \begin{align*}
\bC_{i+1}\fisher\bu_\ell &= (\bI-\bK_i\bH)\bC_i(\bI-\bK_i\bH)^\top\fisher\bu_\ell\\
        &=\bbM_i\bC_i(\bI-\bH^\top(\bH\bC_i\bH^\top+\bGamma)^{-1}\bH\bC_i)\fisher\bu_\ell\\
        &=\bbM_i(\bI- \bK_i\bH)\bC_i\fisher\bu_\ell = \bbM_i^2 \delta_{\ell,i}'\bu_\ell\\
        &=(\bI+\bC_i\fisher)^{-2}\delta_{\ell,i}'\bu_\ell = \delta_{\ell,i}'/{(1+\delta_{\ell,i}')^2}\bu_\ell= \delta_{\ell,i+1}'\bu_\ell.
    \end{align*}
\end{proof}

\Cref{prop: deterministic state eigen basis} defines $h$ state space eigenvectors of~\eqref{eq:GEVstatespace} in terms of observation space eigenvectors of~\eqref{eq: observation GEV} and shows them to be associated with the same  eigenvalues (both zero and nonzero). 
\Cref{prop: eigenvector basis deterministic observation} states that the remaining $n-h$ eigenvalues of \eqref{eq: observation GEV} are zero.
Note that
\eqref{eq:GEVstatespace} must have exactly $r$ nonzero eigenvalues because the nonzero eigenvalues of $(\bC_i\bH^\top)(\bGamma^{-1}\bH)$ must be identical to those of $(\bGamma^{-1}\bH)(\bC_i\bH^\top)$, whereas zero eigenvalues of either matrix (when they exist) may vary only in multiplicity as determined by dimension (see e.g., \cite[Theorem 1.3.22]{HornJohnson2012MatrixAnalysis}). Thus, the remaining $d-h$ eigenvalues that have been left unspecified in \Cref{prop: deterministic state eigen basis} must also be zero. 

Going forward, we will drop the notational distinction between the eigenvalues of~\eqref{eq: observation GEV} and of~\eqref{eq:GEVstatespace}, i.e., we take $\delta'_{\ell,i} = \delta_{\ell,i}$ and constrain the eigenvalue index to $\ell=1,2,\ldots, h$. Let $\bDelta_{1:h}^{(i)}=\diag(\delta_{1,i},\ldots,\delta_{h,i})$ and $\bU=[\bu_1,\dots,\bu_h]$. Then, $\bC_i\fisher\bU = \bU\bDelta_{1:h}^{(i)}$, and the weighted orthonormalization $\bW^\top\bGamma\bW=\bI_n$ implies $\bU^\top\fisher\bU=\bI_h$.

We now define the three fundamental subspaces of EKI in the state space $\R^d$ as specific invariant subspaces of the iteration map for the state space residual, $\bbM_i=(\bI+\bC_i\bH^\top\bGamma^{-1}\bH)^{-1}$,  characterized through appropriately defined spectral projectors.  

 \begin{proposition}\label{prop: det state spec proj}
     Let 
     $\bU_{k:\ell}\in\R^{d\times(\ell-k+1)}$ contain columns $k$ through $\ell$  of $\bU$. Define
       $\bbP = \bU_{1:r}\bU_{1:r}^\top(\fisher)$,  $\bbQ = \bU_{r+1:h}\bU_{r+1:h}^\top(\fisher)$, and  $\bbN = \bI - \bbP-\bbQ$.
    Then $\bbP$, $\bbQ$, and $\bbN$ are spectral projectors for the residual iteration map $\bbM_i$, i.e., $\bbM_i\bbP=\bbP\bbM_i$ and $\bbP^2 = \bbP$, with similar assertions for $\bbQ$ and $\bbN$. $\bbP$, $\bbQ$, and $\bbN$ are complementary in the sense that $\bbP\bbQ=\bbQ\bbN=\bbP\bbN=\boldsymbol{0}$ and $\bbP+\bbQ+\bbN=\bI$
\end{proposition}
\begin{proof}
    The assertions $\bbP^2=\bbP$ and $\bbQ^2 = \bbQ$ and $\bbP\bbQ=\mathbf{0}$ can be verified from their definitions and the weighted orthonormalization condition $\bU^\top\fisher\bU=\bI_h$, so $\bbP$ and $\bbQ$ are both projectors and $\bbP+\bbQ=\bU\bU^\top\fisher$ is also a projector. Then, $\bbN$ is the complementary projector to $\bbP+\bbQ$ with $\Ran(\bbN)=\Ker(\bbP+\bbQ)=\Ker(\bH)$.
    Then,
    \begin{align*}
        \bbM_i^{-1}&(\bU_{1:r}(\bI+\bDelta_{1:r}^{(i)})^{-1}\bU_{1:r}^\top\fisher+\bbQ+\bbN) \\
        &= (\bI+\bC_i\fisher)(\bU_{1:r}(\bI+\bDelta_{1:r}^{(i)})^{-1}\bU_{1:r}^\top\fisher+\bbQ+\bbN)\\
        &= \bU_{1:r}\bU_{1:r}^\top\fisher+\bbQ+\bbN = \bbP +\bbQ+\bbN=\bI,
    \end{align*}
    so $\bbM_i = \bU_{1:r}(\bI+\bDelta_{1:r}^{(i)})^{-1}\bU_{1:r}^\top\fisher+\bbQ+\bbN$. Using this expression for $\bbM_i$ yields
    \begin{align*}
        \bbP\bbM_i &= \bU_{1:r}\bU_{1:r}^\top\fisher\,(\bU_{1:r}(\bI+\bDelta_{1:r}^{(i)})^{-1}\bU_{1:r}^\top\fisher+\bbQ+\bbN)\\
        &= \bU_{1:r}(\bI+\bDelta_{1:r}^{(i)})^{-1}\bU_{1:r}^\top\fisher \\
        &= (\bU_{1:r}(\bI+\bDelta_{1:r}^{(i)})^{-1}\bU_{1:r}^\top\fisher+\bbQ+\bbN)\,\bU_{1:r}\bU_{1:r}^\top\fisher = \bbM_i\bbP.
    \end{align*}
    Similar calculations show that $\bbM_i$ commutes with $\bbQ$ and $\bbN$. 
\end{proof}
The least squares residual can thus be divided into three components lying in complementary subspaces associated with the oblique projectors defined in \Cref{prop: det state spec proj}: $\bomega_i^{(j)} = \bbP\bomega_i^{(j)}+\bbQ\bomega_i^{(j)}+\bbN\bomega_i^{(j)}$. In the next section, we will analyze the convergence behavior of these three residual components.

\begin{remark}\label{rem: trivial state space}
    Any of these three subspaces of state space may be trivial in special cases: if $\Ker(\bH)$ is trivial then so is $\Ran(\bbN)$. If $\Ran(\bC_i)\supset\Ran(\bH^\top)$ then $\Ran(\bbQ)$ is trivial. And again, if $\Ran(\bC_i)\subset\Ker(\bH)$, then $\Ran(\bbP)$ is trivial.
\end{remark}

\paragraph{Related work}
We highlight the relationship between our spectral decompositions and two previous works. First, the work~\cite{bungert2023complete} previously provided a spectral analysis of the ensemble covariance, $\bC_i$, for the continuous-time limit of linear deterministic EKI. This leads to a system of differential algebraic equations describing the evolution of both the eigenvectors and eigenvalues of $\bC_i$. In contrast, we provide a spectral analysis of $\bbM_i$, the iteration map governing the evolution of the least squares residual, whose eigenvectors remain constant and whose eigenvalues satisfy the simple recurrence shown in \Cref{prop: observation GEV behavior}. This enables us to define fundamental subspaces in state space which are invariant under the residual iteration map.
Second, the work~\cite{snyder2022optimal} proposed linear transformations of the state and observation space in data assimilation problems that diagonalize the forward operator and noise covariance, leading to a diagonal Kalman gain. The motivation is that when the Kalman gain must be estimated from particles, localization strategies that shrink off-diagonal elements can be better motivated with respect to these transformed coordinates than the original ones. These transformations can be obtained from our spectral analysis as follows: if $\bC_i$ is assumed to be full rank and interpreted as a prior covariance in the data assimilation setting, and if $\bH$ is assumed to have full column rank so that $h=d$, then one may verify that the transformations $\by' = \bW^\top \by$ and $\bv' = \bDelta^{\frac12}\bU^\top\bC_i^{-1}\bv$ lead to $\by'=\bH'\bv'$ with $\bH'$ diagonal, and that $\beps'=\bW^\top\beps$ is a standard multivariate normal distribution.

\subsubsection{Deterministic EKI: Convergence analysis in state space $\R^d$}\label{sssec: deterministic state convergence}
We draw on the eigenvalue convergence results developed in~\Cref{prop: deterministic eigenvalue bounds} to prove our main result concerning the evolution of the least squares residual of the EKI particles in the state space $\R^d$.

\begin{theorem}\label{thm: deterministic state space convergence}
    For all particles $j = 1,2,\ldots,J$, the following hold:
    \begin{enumerate}[label=(\alph*)]
        \item as $i\to\infty$, $\|\bbP\bomega_i^{(j)}\|=\mathcal{O}(i^{-\frac12})$,
        \item for all $i\geq0$, $\bbQ\bomega_i^{(j)}=\bbQ\bomega_0^{(j)}$, and 
        \item for all $i\geq0$, $\bbN\bomega_i^{(j)}=\bbN\bomega_0^{(j)}$.
    \end{enumerate}
\end{theorem}
\begin{proof}
    The proof is analogous to that of \Cref{thm: deterministic observation particles}.
    Let $\bomega_i^{(j)} = \bbM_{i0}\bomega_0^{(j)}$, where $\bbM_{i0} = \bbM_i\bbM_{i-1}\cdots\bbM_0$ with 
    $\bbM_i = \bU_{1:r}(\bI + \bDelta_{1:r}^{(i)})^{-1}\bU_{1:r}^\top(\fisher) + \bbQ+\bbN$ and  $\bDelta_{1:r}^{(i)}$ denotes the diagonal matrix of the first $r$ nonzero eigenvalues. Thus, with the earlier definition of $\bD_{i0} = \prod_{k=0}^i(\bI+\bDelta_{1:r}^{(k)})^{-1}$, we have
    \begin{align*}
        \bbM_{i0} = \bU_r\bD_{i0}\bU_r^\top(\fisher) + \bbQ + \bbN
    \end{align*}
    Thus, for all $i$, $\bbQ\bomega_i^{(j)} = \bbQ\bbM_{i0}\bomega_0^{(j)} = \bbQ\bomega_0^{(j)}$,  $\bbN\bomega_i^{(j)} = \bbN\bbM_{i0}\bomega_0^{(j)} = \bbN\bomega_0^{(j)}$, and
    \begin{align*}
        \bbP\bomega_i^{(j)} = \bbP\bbM_{i0}\bomega_0^{(j)} = \bU_{1:r}\bD_{i0}\bU_{1:r}^\top(\fisher)\bbP\bomega_0^{(j)}.
    \end{align*}
    Recall $\|\bD_{i0}\|=d_{r,i}$ where $d_{r,i}$ is the largest diagonal element of $\bD_{i0}$ as before. Then, 
    \begin{align*}
        \|\bbP\bomega_i^{(j)}\|\leq d_{r,i}\|\bU_{1:r}\| \, \|\fisher\bU_{1:r}\|\, \|\bbP\bomega_0^{(j)}\|.
    \end{align*}
    Note that for each $j$, $\|\bU_{1:r}\| \, \|\fisher\bU_{1:r}\|\, \|\bbP\bomega_0^{(j)}\|$ is a constant independent of $i$. Since we have shown in the proof of \Cref{thm: deterministic observation particles} that $d_{r,i}=\mathcal{O}(i^{-\frac12})$, we are done.
\end{proof}

\begin{corollary}\label{cor: state particle convergence deterministic} 
    $\displaystyle    \lim_{i\to\infty}\bv_i^{(j)} = \bbP\bH^+\by + \bbQ\bv_0^{(j)} + \bbN\bv_0^{(j)},
    $ for each $j = 1,2,\ldots,J$.
\end{corollary}
Thus, each particle converges to the minimum-norm least squares solution $\bH^+\by$ only in its $\bbP$-component, while its $\bbQ$- and $\bbN$-components remain at their initial values. This is a direct conclusion from \Cref{thm: deterministic state space convergence}.

\paragraph{Related work}
The earlier works~\cite{schillings2017analysis,schillings2018convergence} decompose the state space behavior into two components corresponding to our projection $\bbP$ and a complementary projector under the assumption that $\bH$ is one-to-one, and focus on recovery of the pre-image under $\bH$ of $\by$ rather than the least squares solution~\eqref{eq: ls solution}. The work~\cite{bungert2023complete} assumes $\bC_i$ is full rank and shows that the $j$th particle $\bv_i^{(j)}$ converges to the minimizer of the least-squares objective~\eqref{eq: least squares problem} that is closest in the $\bC_0^{-1}$-norm to its initialization, $\bv_0^{(j)}$.
In contrast, \Cref{thm: deterministic state space convergence} provides the \textit{first results describing convergence of EKI particles to the standard minimum-norm least-squares minimizer}~\eqref{eq: ls solution}. We allow both $\bH$ and $\bC_i$ to be rank-deficient, leading to the definition of three fundamental invariant subspaces of EKI in state space analogous to those previously defined in observation space.

\section{Analysis of stochastic EKI}\label{sec: stochastic analysis}
We now provide an analysis of basic linear \textit{stochastic} EKI (\Cref{alg:EKI} with $\bSigma=\bGamma$). Paralleling our analysis of linear deterministic EKI from \Cref{sec: deterministic analysis}, we begin with stochastic EKI results in observation space (\Cref{ssec: stoch observation space}) before developing related results in state space (\Cref{ssec: stoch state space}).

\subsection{Stochastic EKI: Analysis in observation space $\R^n$}\label{ssec: stoch observation space}
We begin by deriving an idealized data misfit iteration that reflects an idealized covariance update (\Cref{ssec: stoch misfit iteration}). \Cref{sssec: stochastic observation spectral} then provides a spectral analysis of this idealized iteration which distinguishes three fundamental subspaces of stochastic EKI. Convergence behaviors within these subspaces are analyzed in \Cref{sssec: stochastic observation convergence}.

\subsubsection{Stochastic EKI: An idealized data misfit iteration}\label{ssec: stoch misfit iteration} Recall our definition of the data misfit: $\btheta_i^{(j)}=\bH\bv_i^{(j)}-\by$ for $j\leq J$ and $i\geq0$. In stochastic EKI, $\by_i^{(j)}=\by+\beps_i^{(j)}$ in the particle update~\eqref{eq: EKI update}. This yields a misfit iteration similar to \eqref{eq: deterministic misfit iteration} but with a forcing term arising from the stochastic perturbation $\beps_i^{(j)}$:
\begin{align}\label{eq: stochastic misfit iteration}
    \btheta_{i+1}^{(j)} = \bcalM_i\btheta_i^{(j)} + (\bI-\bcalM_i)\beps_{i}^{(j)}
\end{align}
where $\bcalM_i= \bGamma(\bH\bC_i\bH^\top + \bGamma)^{-1}$, defined as before. 
In our analysis of deterministic EKI, we showed that generalized eigenvectors of the pencil $(\bH\bC_i\bH^\top,\bGamma)$ remain constant under the deterministic EKI iteration, enabling us to define fundamental subspaces that are invariant under $\bcalM_i$ across all iterations. This invariance no longer holds for stochastic EKI: to see this, note that under the stochastic misfit iteration~\eqref{eq: stochastic misfit iteration}, the observation space covariance $\bH\bC_i\bH^\top = \sfCov[\btheta_i^{(1:J)}]$ satisfies
\begin{equation}\label{eq: full stoch obs covariance iteration}
\small
    \begin{aligned}
    \bH\bC_{i+1}\bH^\top &= \bcalM_i\bH\bC_i\bH^\top\bcalM_i^\top + \bcalM_i\sfCov[\bH\bv_i^{(1:J)},\beps_i^{(1:J)}](\bI-\bcalM_i)^\top \cdots \\&+(\bI-\bcalM_i)\sfCov[\beps_i^{(1:J)},\bH\bv_i^{(1:J)}]\bcalM_i^\top + (\bI-\bcalM_i)\sfCov[\beps_i^{(1:J)}](\bI-\bcalM_i)^\top.
\end{aligned}
\end{equation}
Relative to its deterministic EKI analogue (cf.\ the proof of \Cref{prop: observation GEV behavior}), eq.~\eqref{eq: full stoch obs covariance iteration} has several additional terms dependent on the realizations of the stochastic perturbations $\beps_i^{(j)}$ at the current iteration that will lead to eigenvectors of $(\bH\bC_i\bH^\top,\bGamma)$ changing from one iteration to the next.
Instead of carefully accounting for these changes, we will base our analytical treatment of stochastic EKI on an idealized covariance iteration which we now motivate and define.

Let $\beps_{0:k}^{(1:J)}=\{\beps_i^{(j)}\}_{i=0,j=1}^{k,J}$ denote the set of all stochastic perturbations through iteration $k$ for all particles $j=1,\ldots,J$. Note that in~\eqref{eq: full stoch obs covariance iteration}, the quantities $\bC_i$, $\bcalM_i$, and $\bv_i^{(j)}$ are all random variables that depend on previous noise realizations $\beps_{0:i-1}^{(1:J)}$. Conditioning on the previous noise and taking the expectation of~\eqref{eq: full stoch obs covariance iteration} with respect to the current ($i$th) noise perturbations yields
\begin{align*}
\E[\bH\bC_{i+1}\bH^\top|\beps_{0:i-1}^{(1:J)}] &= \bcalM_i\bH\bC_i\bH^\top \bcalM_i^\top + (\bI-\bcalM_i)\bGamma(\bI-\bcalM_i)^\top,\nonumber
\end{align*}
which, after substitution and some  rearrangement becomes
\begin{align}
    \E[\bH\bC_{i+1}\bH^\top|\beps_{0:i-1}^{(1:J)}] &= \bGamma - \bGamma(\bH\bC_i\bH^\top + \bGamma)^{-1}\bGamma.\label{eq: stoch covariance conditional iteration}
\end{align}
This motivates the definition of $\tbB_i\in\R^{n\times n}$ for $i\geq0$ as follows:
\begin{align}\label{eq: stoch C definition}
    \tbB_0 = \bH\bC_0\bH^\top \quad \mbox{and}\quad \tbB_{i+1} = \bGamma - \bGamma(\tbB_i + \bGamma)^{-1}\bGamma,
    \quad \mbox{for}\quad i\geq0.
\end{align} 
The matrices $\tbB_i$ satisfy an \textit{idealized covariance iteration}~\eqref{eq: stoch C definition} reflecting the form of the conditional expectation~\eqref{eq: stoch covariance conditional iteration}, and so $\tbB_i$ can be viewed as idealized analogues of $\bH\bC_i\bH^\top$. We then define the idealized misfit iteration map $\widetilde\bcalM_i = \bGamma(\tbB_i+\bGamma)^{-1}$, analogous to $\bcalM_i$, leading to the following \textit{idealized misfit iteration} (analogous to~\eqref{eq: stochastic misfit iteration}):
\begin{align}\label{eq: simplified stochastic misfit iteration}
    \tilde\btheta_{0}^{(j)} = \btheta_0^{(j)}, \qquad \tilde\btheta_{i+1}^{(j)} = \widetilde\bcalM_i\tilde\btheta_i^{(j)} + (\bI-\widetilde\bcalM_i)\beps_{i}^{(j)}, \quad\text{for }i=0,1,\ldots.
\end{align}
In what follows, our analysis of stochastic EKI will treat this idealized misfit iteration and its state-space counterpart, which we will show share favorable properties with their deterministic EKI analogues.

\paragraph{Related work}
We relate the analysis of our idealized iteration~\eqref{eq: simplified stochastic misfit iteration} to earlier works providing analyses of stochastic EKI. 
The works~\cite{blomker2019well,blomker2022continuous} analyze the continuous-time limit of the stochastic iteration, yielding a system of coupled stochastic differential equations governing  individual particle trajectories. 
In this analysis approach, the authors introduce additive covariance inflation in order to prevent the ensemble from collapsing prematurely before converging to a solution~\cite{blomker2019well}.
Along similar lines, we show within our framework in \Cref{prop: 9c upper bound} that $\tbB_i\geq \E[\bH\bC_i\bH^\top]$, so that the idealized iteration~\eqref{eq: simplified stochastic misfit iteration} can be interpreted as reflecting an implicit inflation of covariance. We show that under~\eqref{eq: stoch C definition}, the idealized covariance $\tbB_i$ collapses only in the infinite iteration limit.
Note that by defining $\tbB_i$ so as to reflect the conditional expectation iteration~\eqref{eq: stoch covariance conditional iteration} (as opposed to  explicitly adding a positive covariance inflation term), our approach enables the true stochastic EKI iteration~\eqref{eq: stochastic misfit iteration} to be interpreted as a particle approximation of the idealized iteration~\eqref{eq: simplified stochastic misfit iteration}. 
The results we prove concerning the idealized iteration would therefore exactly describe the behavior of the true stochastic EKI iteration if the empirical covariances of the observation perturbations at each iteration matched their expected values. While this is vanishingly unlikely for any finite ensemble size, large ensembles will more closely approximate our idealized iteration, shedding light into the failure of stochastic EKI to converge when the ensemble size is small (see numerical results in \Cref{sec: numerics}).
Our analysis approach therefore shares some commonalities with mean-field limit analyses of stochastic EKI~\cite{bungert2023complete,ding2021ensemble}, which analyze the algorithm in the infinite ensemble limit, leading to deterministic expressions governing the ensemble statistics (mean and covariance). In contrast, while our idealized covariance iteration expression~\eqref{eq: stoch C definition} is deterministic, we provide new expressions for the idealized misfits of \textit{individual particles}~\eqref{eq: simplified stochastic misfit iteration} retaining the stochastic dependence on  perturbations, and revealing convergence properties for individual particle paths supported by numerical experiments.
Finally, the work~\cite{ghattas2024NonAsymptoticAnalysisEnsemble} provides non-asymptotic analysis of EKI proving that EKI can perform well with small ensembles provided the effective dimension of the inverse problem is small; under these circumstances the true statistics can be well-approximated by a smaller ensemble, and in our setting would lead to the true stochastic EKI closely approximating our idealized iteration with a smaller ensemble.

\subsubsection{Stochastic EKI: Decomposition of observation space $\R^n$}\label{sssec: stochastic observation spectral}
We now provide a spectral analysis of $\widetilde\bcalM_i$ that distinguishes three fundamental subspaces that are invariant under~\eqref{eq: simplified stochastic misfit iteration}. Consider the generalized eigenvalue problem
\begin{align}\label{eq: GEV upper bound}
     \tbB_i \widetilde\bw_{\ell,i} = \tilde\delta_{\ell,i}\bGamma\widetilde\bw_{\ell,i}.
\end{align}
We now show an analogue of~\Cref{prop: observation GEV behavior}:
\begin{proposition}\label{prop: 9c eigenvalue recurrence}
    Let $(\tilde\delta_i,\widetilde\bw)$ be an eigenpair for the pencil $(\tbB_i,\bGamma)$, i.e., satisfying~\eqref{eq: GEV upper bound}. Then, $\widetilde\bw$ is also an eigenvector of the pencil $(\tbB_{i+1},\bGamma)$ with eigenvalue $\tilde\delta_{i+1} = \frac{\tilde\delta_{i}}{1+\tilde\delta_{i}}$.
\end{proposition}
\begin{proof}
From our definition~\eqref{eq: stoch C definition} of the iteration determining $\tbB_i$, we have
\begin{align*}
    \tbB_{i+1} \widetilde \bw &= (\bI-\bGamma(\tbB_i+\bGamma)^{-1})\bGamma\widetilde\bw = \tbB_i(\tbB_i+\bGamma)^{-1}\bGamma\widetilde\bw=\frac{\tilde\delta_{i}}{1+\tilde\delta_{i}}\bGamma\widetilde\bw.
\end{align*}
\end{proof}

As in the deterministic case we now write $\widetilde\bw_\ell=\widetilde\bw_{\ell,i}$ for all $i$.
An analogue of~\Cref{prop: eigenvector basis deterministic observation} concerning the construction of an eigenvector basis for $\R^n$ holds:
\begin{proposition}\label{prop: stochastic observation basis}
    $\tilde\delta_{\ell,0}=0$ implies $\tilde\delta_{\ell,i}=0$ for all $i\geq 1$; and $\tilde\delta_{\ell,0}>0$ implies $\tilde\delta_{\ell,i}>0$ for all $i\geq1$. Let $r$ denote the number of positive eigenvalues of~\eqref{eq: GEV upper bound}. There is a $\bGamma$-orthogonal basis for $\R^n$ comprised of eigenvectors of~\eqref{eq: GEV upper bound}, $\{\widetilde\bw_1,\ldots,\widetilde\bw_n\}$, satisfying
    \begin{enumerate}
        \item $\{\widetilde\bw_1,\ldots,\widetilde\bw_r\}\subset\Ran(\bGamma^{-1}\bH)$ are eigenvectors of~\eqref{eq: GEV upper bound} associated with positive eigenvalues $\tilde\delta_{1,i},\ldots,\tilde\delta_{r,i}$, labeled in non-increasing order at $i=1$, that is, $\tilde\delta_{1,1}\geq \tilde\delta_{2,1}\geq \cdots\geq\tilde\delta_{r,1}>0$. This ordering is preserved for subsequent $i\geq 1$.
        \item if $r<h$, $\{\widetilde\bw_{r+1},\ldots,\widetilde\bw_h\}\subset\Ran(\bGamma^{-1}\bH)$ are eigenvectors of~\eqref{eq: GEV upper bound} associated with zero eigenvalues, and
        \item if $h<n$, $\{\widetilde\bw_{h+1},\ldots,\widetilde\bw_n\}\subset\Ker(\bH^\top)$ are eigenvectors of~\eqref{eq: GEV upper bound} also associated with zero eigenvalues.
    \end{enumerate}
\end{proposition}
The proof is analogous to that of \Cref{prop: eigenvector basis deterministic observation}.

Let $\widetilde\bW = [\widetilde\bw_1,\ldots,\widetilde\bw_n]$, with the normalization $\widetilde\bW^\top\bGamma\widetilde\bW=\bI$. We define spectral projectors of $\widetilde\bcalM_i$ that divide $\R^n$ into three fundamental subspaces of stochastic EKI:
\begin{proposition}\label{prop: stoch meas spec proj}
Let $\widetilde\bW_{k:\ell}\in\R^{n\times(\ell-k+1)}$ denote the $k$-through-$\ell$-th columns of $\widetilde\bW$. Define $\widetilde\bcP = \bGamma\widetilde\bW_{1:r}\widetilde\bW_{1:r}^\top$, $\widetilde\bcQ = \bGamma\widetilde\bW_{r+1:h}\widetilde\bW_{r+1:h}^\top$, and $\widetildeto{\bcQ}{\bcN} = \bGamma\widetilde\bW_{h+1:n}\widetilde\bW_{h+1:n}^\top$.
Then, $\widetilde\bcP$, $\widetilde\bcQ$, and $\widetildeto{\bcQ}{\bcN}$ are complementary spectral projectors associated with the idealized misfit iteration map $\widetilde\bcalM_i$.    
\end{proposition}
The proof is essentially the same as that of \Cref{def: observation projectors}.

As in the deterministic case, the spectral projectors $\widetilde\bcP$, $\widetilde\bcQ$, and $\widetildeto{\bcQ}{\bcN}$ decompose the idealized misfit $\tilde\btheta_i^{(j)}$ into three components exhibiting differing convergence behaviors, which we characterize in the next section.

\subsubsection{Stochastic EKI: Convergence analysis in observation space $\R^n$}\label{sssec: stochastic observation convergence}

We begin by showing that the positive eigenvalues of~\eqref{eq: GEV upper bound} decay at a $1/i$ rate:
\begin{corollary}\label{cor: stochastic eigenvalue rates}
    If $\tilde\delta_{\ell,0}>0$, then for all $i$, $\tilde\delta_{\ell,i} = (\frac1{\tilde\delta_{\ell,0}} + i)^{-1}$.
\end{corollary}
\begin{proof}
    The proof follows by induction with the base case $i=1$ established directly from \Cref{prop: 9c eigenvalue recurrence}.
\end{proof}

We now turn our attention to the behavior of the idealized misfit $\tilde\btheta_{i}^{(j)}$ as $i\to\infty$. Let $\widetilde\bcalM_{ik} = \prod_{j=k}^i\widetilde\bcalM_j$. 
Note that~\eqref{eq: simplified stochastic misfit iteration} implies
\begin{align}\label{eq: mean field misfit from zero}
    \tilde\btheta_{i+1}^{(j)} = \widetilde\bcalM_{i0}\tilde\btheta_0^{(j)} + (\bI-\widetilde\bcalM_i)\beps_i^{(j)} +\sum_{k=0}^{i-1}\widetilde\bcalM_{i,k+1}(\bI-\widetilde\bcalM_k)\beps_k^{(j)}.
\end{align}
We first show a lemma expressing the operators $\widetilde\bcalM_{ik}$ and $(\bI-\widetilde\bcalM_i)$ in terms of eigenvectors and eigenvalues of~\eqref{eq: GEV upper bound}:
\begin{lemma}\label{lem: stoch misfit expressions}
The following hold:
\begin{align}
    \bI - \widetilde\bcalM_i = \sum_{\ell=1}^r\tilde\delta_{\ell,i+1}\bGamma\widetilde\bw_\ell\widetilde\bw_\ell^\top; \quad \widetilde\bcalM_{ik} = \sum_{\ell=1}^r\frac{\tilde\delta_{\ell,i+1}}{\tilde\delta_{\ell,k}}\bGamma\tilde\bw_\ell\tilde\bw_\ell^\top + \widetilde\bcQ + \widetildeto{\bcQ}{\bcN}.
\end{align}
\end{lemma}
\begin{proof}
   Note that $\widetilde\bcalM_i = \bGamma\widetilde\bW(\bI+\widetilde\bDelta_i)^{-1}\widetilde\bW^\top$, 
   so \begin{align*}
    \widetilde\bcalM_i = \sum_{\ell=1}^n\frac1{1+\tilde\delta_{\ell,i}}\bGamma\tilde\bw_\ell\tilde\bw_\ell^\top = \sum_{\ell=1}^r\frac1{1+\tilde\delta_{\ell,i}}\bGamma\tilde\bw_\ell\tilde\bw_\ell^\top + \widetilde\bcQ + \widetildeto{\bcQ}{\bcN}.
   \end{align*}
   Thus, 
   \begin{align*}
       \bI-\widetilde\bcalM_i = \sum_{\ell=1}^r \frac{\tilde\delta_{\ell,i}}{1+\tilde\delta_{\ell,i}}\bGamma\tilde\bw_\ell\tilde\bw_\ell^\top, 
        \quad 
       \widetilde\bcalM_{ik} = \sum_{\ell=1}^r \left(\prod_{j=k}^i\frac1{1+\tilde\delta_{\ell,j}}\right)\bGamma\tilde\bw_\ell\tilde\bw_\ell^\top + \tilde\bcQ + \widetildeto{\bcQ}{\bcN}.
   \end{align*}
   Recall from~\cref{prop: 9c eigenvalue recurrence} that $\tilde\delta_{\ell,i+1} = \frac{\tilde\delta_{\ell,i}}{1+\tilde\delta_{\ell,i}}$, which implies $\frac1{1+\tilde\delta_{\ell,i}} = \frac{\tilde\delta_{\ell,i+1}}{\tilde\delta_{\ell,i}}$. Substituting these relationships into the above expression yields the desired claims.
\end{proof}

Our main result concerns the convergence of the idealized misfit iteration~\eqref{eq: simplified stochastic misfit iteration}. 
\begin{theorem}\label{thm: stochastic observation space convergence}
    For all particles $j = 1,2,\ldots,J$, the following hold:
    \begin{enumerate}[label=(\alph*)]
        \item as $i\to\infty$, $\E[\|\widetilde\bcP\tilde\btheta_i^{(j)}\|]=\mathcal{O}(i^{-\frac12})$, 
        \item for all $i\geq0$, $\widetilde\bcQ\tilde\btheta_i^{(j)} = \widetilde\bcQ\tilde\btheta_0^{(j)}$, and
        \item for all $i\geq 0$, $\widetildeto{\bcQ}{\bcN}\tilde\btheta_i^{(j)} = \widetildeto{\bcQ}{\bcN}\tilde\btheta_0^{(j)}$.
    \end{enumerate}
\end{theorem}
\begin{proof}
    For all $i$, $\widetilde\bcQ$ is a spectral projector of $\widetilde\bcalM_i$ (\Cref{prop: stoch meas spec proj}) and $\Ran(\widetilde\bcQ)\subset\Ker(\bI-\widetilde\bcalM_i)$ (\Cref{lem: stoch misfit expressions}). Thus, applying $\widetilde\bcQ$ to~\eqref{eq: mean field misfit from zero} yields:
    \begin{align*}
        \widetilde\bcQ\tilde\btheta_{i+1}^{(j)} = \widetilde\bcalM_{i0}\widetilde\bcQ\tilde\btheta_0^{(j)} = \widetilde\bcQ\tilde\btheta_0^{(j)}\quad \text{for all }i,
    \end{align*}
    which gives (b). The same argument holds for $\widetildeto{\bcQ}{\bcN}\tilde\btheta_i^{(j)}$, which gives (c). To show (a), note that~\eqref{eq: mean field misfit from zero} and  \Cref{lem: stoch misfit expressions} for $i\geq 1$ leads to:
    \begin{align*}
         \widetilde\bcP\tilde\btheta_{i}^{(j)} = 
        \sum_{\ell=1}^r\frac{\tilde\delta_{\ell,i}}{\tilde\delta_{\ell,0}}\bGamma\tilde\bw_\ell\tilde\bw_\ell^\top\left(\widetilde\bcP\tilde\btheta_0^{(j)}\right) +\sum_{\ell=1}^r \tilde\delta_{\ell,i}\bGamma\tilde\bw_\ell\tilde\bw_\ell^\top \left(\sum_{k=0}^{i-1}\beps_k^{(j)}\right).
    \end{align*}
    Let $c_\ell=\frac1{\tilde\delta_{\ell,0}}$. Then, from \Cref{cor: stochastic eigenvalue rates} we have:
    \begin{align*}
        \widetilde\bcP\tilde\btheta_{i}^{(j)} = \sum_{\ell=1}^r \frac{c_\ell}{i+c_\ell}\bGamma\tilde\bw_\ell\tilde\bw_\ell^\top\!\left(\widetilde\bcP\tilde\btheta_0^{(j)}\right) + \sum_{\ell=1}^r \frac{1}{i+c_\ell}\bGamma\tilde\bw_\ell\tilde\bw_\ell^\top\!\left(\sum_{k=0}^{i-1}\beps_k^{(j)}\right),
    \end{align*}
    so that
    \begin{align*}
        \mathbb{E}[\widetilde\bcP\tilde\btheta_{i}^{(j)}] = \sum_{\ell=1}^r \frac{c_\ell}{i+c_\ell}\bGamma\tilde\bw_\ell\tilde\bw_\ell^\top\left(\widetilde\bcP\tilde\btheta_0^{(j)}\right)\ \mbox{and}\  {\small\cov}\!\left(\widetilde\bcP\tilde\btheta_{i}^{(j)}\right)= \sum_{\ell=1}^r \frac{i}{(i+c_\ell)^2}\bGamma\tilde\bw_\ell\tilde\bw_\ell^\top\bGamma.
    \end{align*}
    Using Jensen's inequality and noting $\mathsf{trace}(\bGamma\tilde\bw_\ell\tilde\bw_\ell^\top\bGamma)=\tilde\bw_\ell^\top\bGamma^{\frac12}\bGamma\bGamma^{\frac12}\tilde\bw_\ell\leq \|\bGamma\|$,  
    {\small
    \begin{equation}\label{Pbnd}
    \E[\|\widetilde\bcP\tilde\btheta_i^{(j)}\|]^2 \leq \E\left[\|\widetilde\bcP\tilde\btheta_i^{(j)}\|^2\right]\leq \mathsf{trace}({\small\cov}\!\left(\widetilde\bcP\tilde\btheta_{i}^{(j)}\right))\leq
        \|\bGamma\|\sum_{\ell=1}^r \frac{i}{(i+c_\ell)^2}\leq \|\bGamma\|\left(\frac{r}{i}\right),
    \end{equation}  }
    which is $\mathcal{O}(i^{-1})$, as $i\rightarrow \infty$. This gives the conclusion for (a). 
\end{proof}

To state an analogue of \Cref{cor: obs particle convergence deterministic} for stochastic EKI, one may verify that the following `idealized particle iteration', 
\begin{align}\label{eq: idealized particle iteration stochastic}
    \tilde\bv_{i+1}^{(j)} &= \tilde\bv_i^{(j)} + \tbB_i\bH^\top (\bH\tbB_i\bH^\top + \bGamma)^{-1} (\by+\beps_i^{(j)}-\bH\tilde\bv_i^{(j)}),
\end{align}
together with the definition $\tilde\btheta_i^{(j)} = \bH\tilde\bv_i^{(j)}-\by\equiv \tilde\bh_i^{(j)}-\by$,
leads to the idealized misfit iteration~\eqref{eq: simplified stochastic misfit iteration}. Then, \Cref{thm: stochastic observation space convergence} yields the following corollary:
\begin{corollary} 
     $\displaystyle   \lim_{i\to\infty}\tilde\bh_i^{(j)} = \widetilde\bcP\by + \widetilde\bcQ\tilde\bh_0^{(j)} + \widetildeto{\bcQ}{\bcN}\tilde\bh_0^{(j)},$ for each $j=1,2,\ldots,J$.
\end{corollary}

\begin{remark}
    \Cref{thm: stochastic observation space convergence} shows that our characterization of EKI's differing convergence behaviors in its fundamental subspaces carries over from deterministic EKI (\Cref{thm: deterministic observation particles}) to stochastic EKI. This understanding of stochastic EKI convergence is new: while the work~\cite{bungert2023complete} showed that convergence of the \textit{ensemble mean} under the mean-field limit can be characterized in terms of two subspaces under the assumption that $\bC_i$ is full rank, \Cref{thm: stochastic observation space convergence} generalizes to allow rank-deficient $\bC_i$ and provides a decomposition of the convergence behaviors of \textit{individual particles} under the idealized iteration~\eqref{eq: simplified stochastic misfit iteration}. 
    We note that while \Cref{thm: stochastic observation space convergence} shows convergence of the expected norm of the observable populated component $\widetilde\bcP\btheta_i^{(j)}$, stronger conclusions can be drawn. Indeed, the proof of \Cref{thm: stochastic observation space convergence} shows that $\widetilde\bcP\tilde\btheta_i^{(j)}\rightarrow \mathbf{0}$ in the mean-square sense. 
    We show in the appendix that $\|\widetilde\bcP\tilde\btheta_i^{(j)}\|\rightarrow 0$  \textit{in probability} at a rate that can be made arbitrarily close to $1/\sqrt{i}$ and that in the $i\to\infty$ limit, $\widetilde\bcP\tilde\btheta_i^{(j)}$ converges to zero \textit{almost surely}.
\end{remark}

\subsection{Stochastic EKI: Analysis in state space $\R^d$}
\label{ssec: stoch state space}
We now analyze the behavior of the particles $\bv_i^{(j)}$ in the state space $\R^d$. In~\Cref{sssec: stoch LS iteration}, we define an idealized residual iteration (the state-space counterpart to the idealized misfit iteration from \Cref{ssec: stoch misfit iteration}). We provide a spectral analysis of this iteration and definitions of the fundamental subspaces in~\Cref{sssec: stoch subspaces}, and prove convergence in \Cref{sssec: stoch state convergence}.

\subsubsection{Stochastic EKI: An idealized least-squares residual iteration}\label{sssec: stoch LS iteration}
We again consider the state space residual of the $j$th particle at the $i$th iteration: $\bomega_i^{(j)} = \bv_i^{(j)} - \bv^*$. Under the stochastic update equation of \Cref{alg:EKI}, the state space residual satisfies:
\begin{align}\label{eq: stoch residual iteration}
    \bomega_{i+1}^{(j)} = \bbM_i\bomega_i^{(j)} + \bK_i\beps_i^{(j)},
\end{align}
where $\bbM_i = (\bI + \bC_i\fisher)^{-1}$ and $\bK_i = \bC_i\bH^\top (\bH\bC_i\bH^\top +\bGamma)^{-1}$ as before. Note that $\bbM_i = \bI - \bK_i\bH$. 
Consider the evolution of the empirical particle covariance $\bC_i \equiv \sfCov[\bv_i^{(1:J)}]$.  We take the expectation of $\bC_{i+1}$ conditioned on previously applied perturbations $\beps_{0:i-1}^{(1:J)}$ (analogous to~\eqref{eq: stoch covariance conditional iteration}):
\begin{align*}
\E[\bC_{i+1}|\beps_{0:i-1}^{(1:J)}] = \bbM_i\bC_i\bbM_i^\top + \bK_i\bGamma\bK_i^\top =\bC_i-\bC_i\bH^\top(\bH\bC_i\bH^\top +\bGamma)^{-1}\bH\bC_i =\bbM_i\bC_i.
\end{align*}
In what follows, we define
\begin{align}\label{eq: G def}
    \tbC_0 = \bC_0 \quad \text{and} \quad \tbC_{i+1} = \widetilde\bbM_i \tbC_i,  \quad \text{with} \quad\widetilde\bbM_i = (\bI + \tbC_i\fisher)^{-1}.
\end{align}
We now define $\widetilde\bK_i = \tbC_i\bH^\top(\bH\tbC_i\bH^\top + \bGamma)^{-1}$ and consider the iteration:
\begin{align}\label{eq: stoch simplified LS iteration}
    \tilde\bomega_0^{(j)} = \bomega_0^{(j)}, \qquad \tilde\bomega_{i+1}^{(j)} = \widetilde\bbM_i\tilde\bomega_i^{(j)} + \widetilde\bK_i\beps_i^{(j)}.
\end{align}
One may verify that \eqref{eq: G def} leads to \eqref{eq: stoch C definition} with $\tbB_i = \bH\tbC_i\bH^\top$.
The expression~\eqref{eq: stoch simplified LS iteration} is therefore a state-space analogue of the idealized iteration~\eqref{eq: simplified stochastic misfit iteration} in the observation space. 
As in the observation space, the idealized covariance iteration in the state space~\eqref{eq: G def} reflects the conditional expectation of the true state-space covariance iteration, allowing the true iteration~\eqref{eq: stoch residual iteration} to be interpreted as a particle approximation to the idealized iteration~\eqref{eq: stoch simplified LS iteration}.

\subsubsection{Stochastic EKI: Decomposition of state space $\R^d$}\label{sssec: stoch subspaces}
To analyze the convergence of the idealized state space residual iteration~\eqref{eq: stoch simplified LS iteration}, we consider the following eigenvalue problem:
\begin{align}\label{eq: stoch state eigenvalue}
    \tbC_i\fisher\tilde\bu_{\ell,i} = \tilde\delta_{\ell,i}'\tilde\bu_{\ell,i}.
\end{align}
We now show how the eigenvectors and eigenvalues of~\eqref{eq: stoch state eigenvalue} are related to those of~\eqref{eq: GEV upper bound}:
\begin{proposition}\label{prop: stoch state eigenvectors}
    Let $\{\widetilde\bw_\ell\}_{\ell=1}^n$ denote eigenvectors of~\eqref{eq: GEV upper bound} ordered as described in \Cref{prop: stochastic observation basis}, and recall that the leading $r$ eigenvectors correspond to positive $\tilde\delta_{\ell,i}$. For $\ell=1,\ldots,r$, define $\tilde\bu_\ell=\frac1{\delta_{\ell,i}}\tbC_i\bH^\top\widetilde\bw_\ell$. Then, for all $\ell\leq h$, $\tilde\bu_\ell$ is an eigenvector of~\eqref{eq: stoch state eigenvalue} with eigenvalue $\tilde\delta_{\ell,i}'=\tilde\delta_{\ell,i}$. Conversely, for all $\ell\leq h$, $\widetilde\bw_\ell=\bGamma^{-1}\bH\tilde\bu_\ell$.
\end{proposition}
The proof is analogous to that of \Cref{prop: deterministic state eigen basis}. Going forward we now drop the notational distinction between eigenvalues of~\eqref{eq: GEV upper bound} and of~\eqref{eq: stoch state eigenvalue}, taking $\tilde\delta_{\ell,i}'=\tilde\delta_{\ell,i}$ and constraining the eigenvalue index to $\ell\leq h$. 
We can now define spectral projectors of the idealized state-space residual iteration map $\widetilde\bbM_i$:
\begin{proposition}\label{prop: stoch state spec proj}
    Let $\tilde\bU_{k:\ell}\in\R^{d\times (\ell-k+1)}$ denote the $k$-through-$\ell$th columns of $\tilde\bU$. 
    Define $\widetilde\bbP = \tilde\bU_{1:r}\tilde\bU_{1:r}^\top\fisher$, $\widetilde\bbQ = \tilde\bU_{r+1:h}\tilde\bU_{r+1:h}^\top\fisher$, and $\widetilde\bbN = \bI - \widetilde\bbP - \widetilde\bbQ$.
    Then, $\widetilde\bbP$, $\widetilde\bbQ$, and $\widetilde\bbN$ are spectral projectors associated with $\widetilde\bbM_i$.
\end{proposition}
The proof is the same as that of \Cref{prop: det state spec proj}.

\subsubsection{Stochastic EKI: Convergence analysis in state space $\R^d$}\label{sssec: stoch state convergence}
We now consider the behavior of the state space residual $\tilde\bomega_i^{(j)}$ as $i\to\infty$. Let $\widetilde\bbM_{ik} = \prod_{j=k}^i\widetilde\bbM_j$. Note that \eqref{eq: stoch simplified LS iteration} implies:
\begin{align}\label{eq: stoch simplified LS from zero}
    \tilde\bomega_{i+1}^{(j)} = \widetilde\bbM_{i0}\tilde\bomega_0^{(j)} + \widetilde\bK_i\beps_i^{(j)} + \sum_{k=0}^{i-1}\widetilde\bbM_{i,k+1}\widetilde\bK_k\beps_k^{(j)}.
\end{align}
We first show a lemma expressing $\widetilde\bbM_{ik}$ and $\widetilde\bK_i$ in terms of the spectral projectors defined in \Cref{prop: stoch state spec proj} and the eigenvectors of \eqref{eq: stoch state eigenvalue} and \eqref{eq: GEV upper bound}.
\begin{lemma}\label{lem: stoch state M K expressions}
    The following hold:
    \begin{align}\label{eq: stoch tilde M K expressions}
        \widetilde\bbM_{ik} = \sum_{\ell=1}^r \frac{\tilde\delta_{\ell,i+1}}{\tilde\delta_{\ell,k}}\tilde\bu_\ell\tilde\bu_\ell^\top\fisher + \widetilde\bbQ + \widetilde\bbN, \qquad \widetilde\bK_i = \sum_{\ell=1}^r \tilde\delta_{\ell,i+1}\tilde\bu_\ell\tilde\bw_\ell^\top.
    \end{align}
\end{lemma}
\begin{proof}
    The fact that $\widetilde\bbM_i = \sum_{\ell=1}^r\frac1{1+\tilde\delta_{\ell,i}}\tilde\bu_\ell\tilde\bu_\ell^\top\fisher + \widetilde\bbQ + \widetilde\bbN$ can be verified directly by multiplying out the given expression with $\widetilde\bbM_i^{-1}=\bI + \tbC_i\fisher$ (see the similar calculation in the proof of \Cref{prop: deterministic state eigen basis}). 
    To obtain the expression for $\widetilde\bK_i$, recall $\widetilde\bK_i\bH = (\bI-\widetilde\bbM_i)$, and that $\widetilde\bw_\ell = \bGamma^{-1}\bH\tilde\bu_\ell$ for $\ell\leq r\leq h$ (from \Cref{prop: stoch state eigenvectors}), and that $\frac1{1+\tilde\delta_{\ell,i}} = \frac{\tilde\delta_{\ell,i+1}}{\tilde\delta_{\ell,i}}$. Next,
    \begin{align*}
        \widetilde\bbM_{ik} = \sum_{\ell=1}^r \prod_{j=k}^i \frac{\tilde\delta_{\ell,j+1}}{\tilde\delta_{\ell,j}}\tilde\bu_\ell\tilde\bu_\ell^\top\fisher + \widetilde\bbQ + \widetilde\bbN.
    \end{align*}
    Noting that $\prod_{j=k}^i \frac{\tilde\delta_{\ell,j+1}}{\tilde\delta_{\ell,j}} = \frac{\tilde\delta_{\ell,i+1}}{\tilde\delta_{\ell,k}}$ yields the result.
\end{proof}

We can now prove our main result. 
\begin{theorem}\label{thm: stoch state convergence}
    For all particles $j=1,2,\ldots,J$, the following hold:
    \begin{enumerate}[label=(\alph*)]
        \item as $i\to\infty$, $\E[\|\widetilde\bbP\tilde\bomega_i^{(j)}\|]=\mathcal{O}(i^{-\frac12})$,
        \item for all $i\geq 0$, $\widetilde\bbQ\tilde\bomega_i^{(j)} = \widetilde\bbQ\tilde\bomega_0^{(j)}$, and 
        \item for all $i\geq 0$, $\widetilde\bbN\tilde\bomega_i^{(j)} = \widetilde\bbN\tilde\bomega_0^{(j)}$.
    \end{enumerate}
\end{theorem}
\begin{proof}
    The argument for statements (b) and (c) is essentially the same as the argument for statements (b) and (c) of \Cref{thm: stochastic observation space convergence}. For (a): we apply $\bbP$ to $\tilde\bomega_i^{(j)}$ \eqref{eq: stoch simplified LS from zero} and substitute our expressions from \Cref{lem: stoch state M K expressions} to obtain:
    \begin{align*}
        \widetilde\bbP\tilde\bomega_{i}^{(j)} &= \sum_{\ell=1}^r \frac{\tilde\delta_{\ell,i}}{\tilde\delta_{\ell,0}}\tilde\bu_\ell\tilde\bu_\ell\fisher \tilde\bomega_0^{(j)} + \sum_{k=0}^{i-1}\sum_{\ell=1}^r\tilde\delta_{\ell,i}\tilde\bu_\ell\tilde\bw_\ell^\top\beps_k^{(j)}\\
        &= \sum_{\ell=1}^r \frac{c_\ell}{i+c_\ell}\tilde\bu_\ell\tilde\bu_\ell^\top\fisher\tilde\bomega_0^{(j)} + \sum_{\ell=1}^r\tilde\bu_\ell\tilde\bw_\ell^\top \frac1{i+c_\ell}\sum_{k=0}^{i-1}\beps_k^{(j)},
    \end{align*}
    where $c_\ell = \frac1{\tilde\delta_{\ell,0}}$ as before and we have used \Cref{cor: stochastic eigenvalue rates} to get the second line. 
    We follow a similar argument as for the bound on $\|\widetilde\bcP\tilde\btheta_i^{(j)}\|$ in the proof of \Cref{thm: stochastic observation space convergence}, using the fact that  $\mathsf{trace}(\tilde\bu_\ell\tilde\bu_\ell^\top)=\tilde\bu_\ell^\top(\bH^\top\bGamma^{-1}\bH)^{\frac12}(\bH^\top\bGamma^{-1}\bH)^{\dagger}(\bH^\top\bGamma^{-1}\bH)^{\frac12}\tilde\bu_\ell\leq \|(\bH^\top\bGamma^{-1}\bH)^{\dagger}\|$ and $\tilde\bw_\ell^\top\bGamma\tilde\bw_\ell=1$, 
 to obtain  
$$    
\E[\|\widetilde\bbP\tilde\bomega_i^{(j)}\|]^2 \leq \E\left[\|\widetilde\bbP\tilde\bomega_i^{(j)}\|^2\right]\leq 
\|(\bH^\top\bGamma^{-1}\bH)^{\dagger}\|\sum_{\ell=1}^r \frac{i}{(i+c_\ell)^2}\leq \|(\bH^\top\bGamma^{-1}\bH)^{\dagger}\|\left(\frac{r}{i}\right),
$$
 which is $\mathcal{O}(i^{-1})$, as $i\rightarrow \infty$. This gives the conclusion of (a). 
\end{proof}

\begin{corollary}
    Under the idealized particle iteration~\eqref{eq: idealized particle iteration stochastic},  
    \begin{align*}
\lim_{i\to\infty}\tilde\bv_i^{(j)} = \widetilde\bbP\bH^+\by + \widetilde\bbQ\tilde\bv_0^{(j)} + \widetilde\bbN\tilde\bv_0^{(j)}, \quad \mbox{ for each } j=1,2,\ldots,J.
    \end{align*}
\end{corollary}

\section{Numerical illustration}\label{sec: numerics}
We construct an illustrative example with $n=500$ observations of $d=1000$ states, with randomly generated $\bH,\bGamma,\by$, and $\bv_0^{(1:J)}$ carefully constructed such that all fundamental subspaces are non-trivial\footnote{see \Cref{rem: trivial obs space,rem: trivial state space}}, as follows: $\bGamma\in\R^{n\times n}$ is a random full-rank symmetric positive definite matrix, and $\bH\in\R^{n\times d}$ is constructed so that $\Ker(\bH)$ and $\Ker(\bH^\top)$ are both non-trivial. The observations $\by$ are generated by applying $\bH$ to a random vector in $\R^d$ and adding noise drawn from the normal distribution with mean zero and covariance $\bGamma$. The ensemble is initialized with $J$ random particles that typically have nonzero components in both $\Ran(\bH^\top)$ and $\Ker(\bH)$, but crucially, do \textit{not} contain  $\Ran(\bH^\top)$ within their span. For both deterministic and stochastic EKI, we report results using both a large ensemble with $J = 5000$ particles as well as a small ensemble with $J = 10$ particles. Our code for this example is available at \texttt{https://github.com/elizqian/eki-fundamental-subspaces}.

\begin{figure}
    \centering
    \includegraphics[width=0.9\linewidth]{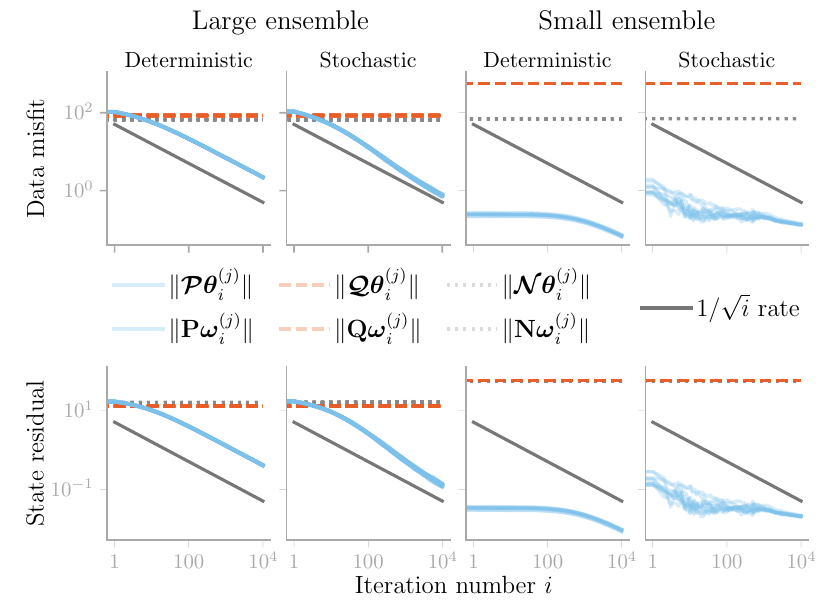}
    \caption{Evolution of particle misfit/residual components under deterministic and stochastic variants of EKI for a randomly generated problem with dimensions $n = 500$ and $d = 1000$. The large ensemble uses $J = 5000$ particles; the small ensemble uses $J = 10$.}
    \label{fig: rate illustration}
\end{figure}

\Cref{fig: rate illustration} illustrates the convergence behavior of both the observation space misfit and the state space residual of the EKI particles under both the deterministic and stochastic EKI dynamics. For both deterministic and stochastic EKI, the projectors are computed from the eigenvalue problems defined by the initial ensemble. The magnitude of each of the components is then plotted vs.\ iteration index $i$ for $i$ up to $i_{\rm max} = 10^4$. This choice of $i_{\rm max}$ allows us to show that the asymptotic behavior of each of the EKI instances reflects our analytical results, but in practice `early stopping' of EKI, e.g., based on the discrepancy principle, may be utilized to avoid overfitting to the data~\cite{iglesias2013ensemble,schillings2017analysis}.

In \Cref{fig: rate illustration}, solid blue lines show the misfit/residual components projected by $\bcP$/$\bbP$ onto observable and populated space. For deterministic EKI, the ensemble size has no effect on the asymptotic convergence rate of this component, but ensemble size does affect \textit{pre}-asymptotic behavior: the small ensemble takes many more iterations to arrive at the asymptote than the large ensemble. For stochastic EKI with a large ensemble, we observe convergence close to the expected $1/\sqrt{i}$ rate in this space, but stochastic EKI with a small ensemble fails to converge. This is because with only $J=10$ particles the true iterations~\cref{eq: stochastic misfit iteration,eq: stoch residual iteration} are far from the their idealized counterparts~\cref{eq: simplified stochastic misfit iteration,eq: stoch simplified LS iteration}, whereas with $J = 5000$ particles the true iteration is closer to the idealized iteration we analyze. Dashed orange lines show the misfit/residual components projected by $\bcQ$/$\bbQ$ onto observable and unpopulated space: because there are no particles in this space, these components remain constant. Dotted gray lines show the misfit/residual components projected by $\bcN$/$\bbN$ onto the unobservable space, which also remain constant. We note that in the measurement space, the size of the unobservable $\bcN$-component of the misfit is independent of ensemble size because this component is exactly the component of the measured data $\by$ lying in $\Ker(\bH)$, whereas the size of the observable populated $\bcP$ [unpopulated $\bcQ$] component is larger [smaller] for the large ensemble, because more directions are populated. In state space, trends for observable populated $\bbP$ and observable unpopulated $\bbQ$ components are similar. However, in state space the unobservable $\bbN$ component is larger for the large ensemble because the large ensemble can span more directions of $\Ker(\bH)$. 

\section{Conclusions}\label{sec: conclusions}
The work presented here offers a new analysis of the behavior of both deterministic and stochastic versions of basic EKI for linear observation operators, interpreting EKI convergence properties in terms of ``fundamental subspaces” analogous to Strang's fundamental subspaces of linear algebra.  Our analysis directly examines the discrete EKI iteration and defines six fundamental subspaces across both observation and state spaces that are distinguished in terms of convergence behaviour. This approach confirms convergence rates previously derived elsewhere for continuous-time limits and yields new results that describe EKI behavior in terms of these fundamental subspaces.  Our analysis is the first to illuminate the relationship between EKI solutions and the standard minimum-norm weighted least squares solution: EKI particles converge to the standard solution only in the observable and populated fundamental subspace, while particle components in the unpopulated observable space and unobservable space remain at their initialized values.

We note that although the EKI iteration itself depends on the iteration index $i$, the fundamental subspaces of EKI that we describe are independent of the iteration, and convergence rates within these subspaces are governed by the magnitude of eigenvalues related to these invariant subspaces. Similar circumstances occur in classical iterative methods for solving linear systems (e.g., see the discussion of Richardson iteration  in \cite[Sect. 6.1.3] {bjorck2024numMeth4LSprob}), suggesting directions for further EKI methodological development and analysis that exploit these connections. For example, EKI variants based on convergence acceleration for classical iterative methods may be pursued as an alternative to convergence acceleration based on covariance inflation. 
Future work will investigate these ideas further.

\section*{Acknowledgments}
Work by EQ was supported in parts by the US Department of Energy Oﬃce of Science Energy Earthshot Initiative as part of the ‘Learning reduced models under extreme data conditions for design and rapid decision-making in complex systems’ project under award number DE-SC0024721, and by the Air Force Oﬃce of Scientiﬁc Research (AFOSR) award FA9550-24-1-0105 (Program Oﬃcer Dr. Fariba Fahroo). CB was supported in part by the National Science Foundation under DMS-2318880.

\appendix

\section{Additional results concerning stochastic EKI}

\begin{proposition}\label{prop: 9c upper bound}
        For $i\geq 0$, $\E[\bH\bC_i\bH^\top]\leq\tbB_i$, with respect to  L\"owner ordering. 
\end{proposition}
\begin{proof}
We proceed by induction. The statement is trivially true at $i = 0$ since $\E[\bH\bC_0\bH^\top]=\bH\bC_0\bH^\top= \tbB_0$. 
Now take the expectation of~\eqref{eq: stoch covariance conditional iteration} over all stochastic perturbations and use the Law of Iterated Expectation to obtain
\begin{align}\label{eq: 9c starting point}
\E[\bH\bC_{i+1}\bH^\top]=&\E\left[\E[\bH\bC_{i+1}\bH^\top|\beps_{0:i-1}^{(1:J)}]\right] =\mathbb{E}[\bGamma - \bGamma(\bH\bC_i\bH^\top + \bGamma)^{-1}\bGamma]\\
    =&\bGamma -\bGamma\,\mathbb{E}[(\bH\bC_i\bH^\top + \bGamma)^{-1}] \bGamma
    \leq \bGamma -\bGamma(\mathbb{E}[\bH\bC_i\bH^\top] + \bGamma)^{-1} \bGamma,\nonumber
\end{align}
where the inequality comes from the convexity of the inverse within the family of positive definite matrices.
The inductive hypothesis together with~\eqref{eq: 9c starting point} implies that
\begin{align*}
\E[\bH\bC_{i+1}\bH^\top]\leq
 \bGamma -\bGamma(\tbB_i + \bGamma)^{-1} \bGamma = \tbB_{i+1},
\end{align*}
where the inequality follows because the function $f(\mathbf{X}) = \bGamma-\bGamma(\mathbf{X}+\bGamma)^{-1}\bGamma$ is non-decreasing with respect to $\mathbf{X}$ in the L\"owner ordering.
\end{proof}

\begin{proposition}
 \label{stochasticPconv} $\|\widetilde\bcP\tilde\btheta_i^{(j)}\|\rightarrow 0$  \textit{in probability} at a rate that can be made arbitrarily close to $1/\sqrt{i}$. That is, $\displaystyle \lim_{i\rightarrow \infty}\mathsf{Prob}\left(\|\widetilde\bcP\tilde\btheta_i^{(j)}\|\geq \frac{\varepsilon}{i^p}\right)=0$ for all $0<p<\frac12$ and all $\varepsilon>0$.  In the $i\to\infty$ limit, $\widetilde\bcP\tilde\btheta_i^{(j)}$ converges to zero 
\textit{almost surely}.
\end{proposition}
\begin{proof}
Note first that for any $\varepsilon>0$, the event $\|\widetilde\bcP\tilde\btheta_i^{(j)}\|\geq \frac{\varepsilon}{i^p}$ is identical to the event $\|\widetilde\bcP\tilde\btheta_i^{(j)}\|^{2}\geq \frac{\varepsilon^{2}}{i^{2p}}$ and so by the Markov inequality and \eqref{Pbnd},
{\small 
$$
\mathsf{Prob}\left(\|\widetilde\bcP\tilde\btheta_i^{(j)}\|\geq \frac{\varepsilon}{i^p}\right)=\mathsf{Prob}\left(\|\widetilde\bcP\tilde\btheta_i^{(j)}\|^{2}\geq \frac{\varepsilon^{2}}{i^{2p}}\right)\leq \frac{i^{2p}}{\varepsilon^{2}}
\E\left[\|\widetilde\bcP\tilde\btheta_i^{(j)}\|^2\right] \leq 
\|\bGamma\|\left(\frac{r}{\varepsilon^{2}}\right)\cdot i^{2p-1}.
$$ }
So, $\displaystyle \lim_{i\rightarrow \infty}\mathsf{Prob}\left(\|\widetilde\bcP\tilde\btheta_i^{(j)}\|\geq \frac{\varepsilon}{i^p}\right)=0$ for all $0<p<\frac12$, establishing the first assertion. 
For the second assertion, rewrite $\widetilde\bcP\tilde\btheta_{i}^{(j)}$ as
$$
\widetilde\bcP\tilde\btheta_{i}^{(j)} = 
 \frac1i\sum_{\ell=1}^r \frac{i \, c_\ell}{i+c_\ell}\bGamma\tilde\bw_\ell\tilde\bw_\ell^\top\!\left(\widetilde\bcP\tilde\btheta_0^{(j)}\right) +  \sum_{\ell=1}^r 
\frac{i}{i+c_\ell} \bGamma\tilde\bw_\ell\tilde\bw_\ell^\top\!\left(\frac1i\sum_{k=0}^{i-1}\beps_k^{(j)}\right).
$$
The first term is non-stochastic and converges to $\mathbf{0}$ as $i\rightarrow \infty$. The second term involves the average of a sequence of \emph{i.i.d.}~random vectors: $\mathbf{m}_i^{(j)}=\frac1i\sum_{k=0}^{i-1}\beps_k^{(j)}$, which by \emph{the strong law of large numbers} converges almost surely to $\mathbf{0}$, the common mean of $\beps_k^{(j)}$: $\mathbf{m}_i^{(j)}\stackrel{a.s.}{\rightarrow}\mathbf{0}$. Hence, $\widetilde\bcP\tilde\btheta_{i}^{(j)}$ also converges almost surely to $\mathbf{0}$ as $i\rightarrow \infty$. 
\end{proof}

\begin{proposition}
     \label{HGamH conv}
$\mathsf{Prob}(\bH\bC_i\bH^\top\leq \varepsilon \,\bGamma)\geq 1-\frac{n}{\varepsilon}\frac1i
$ for $\varepsilon>0$. As a consequence,
$\displaystyle \lim_{i\rightarrow\infty}\mathsf{Prob}\left(\bH\bC_i\bH^\top\leq \frac1{i^p} \bGamma\right)=1$ for all $0<p<1$, so  $\bH\bC_i\bH^\top\rightarrow \mathbf{0}$ in probability at a rate that can be made arbitrarily close to $\frac1i$. Likewise, $\bcalM_i\rightarrow \bI$ in probability as $i\rightarrow \infty$.
\end{proposition}

\begin{proof}
A matrix version of the Markov inequality (see e.g., \cite[Theorem 12]{ahlswede2002strong}) can be stated as: if $\mathbf{X}$ is a random symmetric matrix with finite expectation that is positive semidefinite almost surely (meaning  $\mathsf{Prob}(\mathbf{X}\geq \mathbf{0})=1$), then for any $\varepsilon>0$, {\small $\mathsf{Prob}(\mathbf{X}\leq \varepsilon\bI)\geq 1-\frac1{\varepsilon}\mathsf{trace}(\mathbb{E}[\mathbf{X}])$}.  Choosing $\mathbf{X}=\bGamma^{-\frac12}\bH\bC_i\bH^\top\bGamma^{-\frac12}$, we have $\mathbf{X}\leq \varepsilon\bI$ if and only if {\small $\bH\bC_i\bH^\top\leq \varepsilon\bGamma$}, so 
$
\mathsf{Prob}(\bH\bC_i\bH^\top\leq \varepsilon \bGamma)=\mathsf{Prob}(\mathbf{X}\leq \varepsilon\bI)$.
Combining \Cref{prop: 9c upper bound} with equation \eqref{eq: GEV upper bound} and \Cref{cor: stochastic eigenvalue rates} gives 
$$
\mathsf{trace}(\mathbb{E}[\bGamma^{-\frac12}\bH\bC_i\bH^\top\bGamma^{-\frac12}])\leq \mathsf{trace}(\bGamma^{-\frac12}\tbB_i\bGamma^{-\frac12})=\sum_{\ell=1}^n \tilde\delta_{\ell,i}=\sum_{\ell=1}^n\frac{1}{\tilde\delta_{\ell,0}+i}\leq \frac{n}{i},
$$
which leads to the first bound.  The second assertion follows from choosing $\varepsilon=\frac1{i^p}$ and taking the limit $i\rightarrow \infty$. 
This implies, in particular, convergence of $\bH\bC_i\bH^\top\rightarrow \mathbf{0}$ in probability at a rate of $\frac1{i^p}$. Observe that  $\bcalM_i=\bGamma(\bH\bC_i\bH^\top+\bGamma)^{-1}=\bI-\bH\bC_i\bH^\top(\bH\bC_i\bH^\top+\bGamma)^{-1}$ and so, $\bcalM_i$ is a continuous function of $\bH\bC_i\bH^\top$. Since  $\bH\bC_i\bH^\top\rightarrow \mathbf{0}$ in probability, we have that 
$\bcalM_i\rightarrow \bI$ in probability as $i\rightarrow \infty$ as well. 
\end{proof}

\bibliographystyle{plain}
\bibliography{refs}
\end{document}